\documentclass{amsart}

\usepackage{amssymb}
\usepackage{geometry}
\usepackage{color}
\usepackage{graphicx}

\usepackage{comment}

\newtheorem{thm}{Theorem}[section]
\newtheorem{cor}[thm]{Corollary}
\newtheorem{lem}[thm]{Lemma}
\newtheorem{prop}[thm]{Proposition}
\theoremstyle{definition}
\newtheorem{defn}[thm]{Definition}
\newtheorem{rem}[thm]{Remark}

\newtheorem*{ack}{Acknowledgments}

\usepackage[colorlinks=true,urlcolor=blue, citecolor=red,linkcolor=blue,linktocpage,pdfpagelabels, bookmarksnumbered,bookmarksopen]{hyperref}
\usepackage[hyperpageref]{backref}

\numberwithin{equation}{section}

\author[Brasco]{Lorenzo Brasco}
\address[L.\ Brasco]{Dipartimento di Matematica e Informatica
\newline\indent
Universit\`a degli Studi di Ferrara
\newline\indent
Via Machiavelli 35, 44121 Ferrara, Italy}
\email{lorenzo.brasco@unife.it}

\author[Gonz\'alez]{Mar\'ia del Mar Gonz\'alez}
\address[M.\ Gonz\'alez]{Departamento de Matem\'aticas
\newline\indent
Universidad Aut\'onoma de Madrid and ICMAT\\
\newline\indent
Campus de Cantoblanco, 28049 Madrid, Spain}
\email{mariamar.gonzalezn@uam.es}

\author[Ispizua]{Mikel Ispizua}
\address[M.\ Ispizua]{Departamento de Matem\'aticas
\newline\indent
Universidad Aut\'onoma de Madrid\\
\newline\indent
Campus de Cantoblanco, 28049 Madrid, Spain}
\email{mikel.ispizua@uam.es}

\title{A Steklov version of the torsional rigidity}


\begin{document}

\begin{abstract}
Motivated by the connection between the first eigenvalue of the Dirichlet-Laplacian and the torsional rigidity, the aim of this paper is to find a physically coherent and mathematically interesting new concept for boundary torsional rigidity, closely related to the Steklov eigenvalue. From a variational point of view, such a new object corresponds to the sharp constant for the trace embedding of $W^{1,2}(\Omega)$ into $L^1(\partial\Omega)$. We obtain various equivalent variational formulations, present some properties of the state function and obtain some sharp geometric estimates, both for planar simply connected sets and for convex sets in any dimension.
\end{abstract}

\maketitle
\begin{center}
\begin{minipage}{10cm}
\small
\tableofcontents
\end{minipage}
\end{center}

\section{Introduction}

\subsection{Background}
In this paper we introduce a boundary version for the torsional rigidity which is
modeled on the trace Sobolev embedding $W^{1,2}(\Omega)\hookrightarrow L^1(\partial\Omega)$ for an open bounded set $\Omega\subset \mathbb R^N $, with Lipschitz boundary. This will be also closely related to the Steklov eigenvalue problem. Before we give the precise definitions in Section \ref{subsection:definitions}, let us briefly recall some facts about the usual {\it torsional rigidity} and the {\it Steklov eigenvalue problem}.
\vskip.2cm\noindent
Let us consider an isotropic elastic cylindrical beam whose cross-section is represented by an open bounded simply connected set $\Omega$ in $\mathbb{R}^2$, with boundary $\partial\Omega$. The longitudinal axis of the beam is the $z$-axis, with the cross-sections perpendicular to it. The formulation of  the so-called {\it torsion problem} began at the end of the 18th century, provided by Coulomb, for bars with circular cross section. In this case, it is assumed that the cross-sections rotate around $z$-axis as a rigid body under an applied torque. Nevertheless, because of the complex stress distribution in the bar, this does not happen when the cross-sections do not have circular symmetry and the problem must be reformulated.

The correct statement for beams of general shape was given by the French mechanician and mathematician A. B. de Saint-Venant, who proposed that the deformation in a twisted beam consists in two phenomena: the rotation of the projections in the $x,y$ plane of the cross-sections as rigid bodies, as in the circular case; and the \emph{warping}, equal for all the cross-sections, which do not remain plane after rotation. The distribution of stress generated in the beam due to a an applied torque is determined by the \emph{stress function}, denoted by $v_\Omega$, which satisfies
\begin{equation}
\label{torsion_original}
	\left\{\begin{array}{rcll}
		-\Delta v_\Omega&=&1,& \text{ in }\Omega,\\
	 v_\Omega &=&0,& \text{on }\partial\Omega.
		\end{array}
	\right.
\end{equation}

The total resultant torque due to this stress function is called \emph{torsional rigidity} and can be expressed as

\begin{equation*}
	\mathcal T(\Omega)=\int_{\Omega}v_\Omega\,dx,
\end{equation*}
or equivalently,
\begin{equation*}\label{torsional_originial_variational}
	\mathcal T(\Omega)=\sup_{v|_{\partial\Omega}=0 } \frac{\displaystyle\left(\int_{\Omega}v\,dx\right)^2}{\displaystyle\int_{\Omega}|\nabla v|^2\,dx}.
\end{equation*}
We refer the reader to \cite{LaLi} for a thorough description of this physical model.
\par
Saint-Venant himself stated in 1856 that the simply connected cross-section with maximal torsional rigidity is the circle. In other words, he conjectured the validity of the following inequality
\begin{equation*}
\label{SV}
\mathcal T(\Omega)\le \mathcal T(\Omega^*),
\end{equation*}
where $\Omega^*$ is any circle having the same area as $\Omega$.
This statement was rigorously established, almost a hundred years latter, by G. P\'{o}lya using a symmetrization method \cite{polya1948torsional} and by E. Makai \cite{makai666}. The latter also established the identification of equality cases, by means of a clever quantitative improvement of the inequality \eqref{SV}.
\par
The torsional problem has been extensively studied for its mathematical interest, apart of his obvious physical importance: in addition to the classical book \cite{PolyaandSzego}, we refer for example to \cite{HP, KJ, payne1961some, Polya-Weinstein}, for classical results, as well as to \cite{banuelos2002torsional,BBP,van2016polya,vanMa,BrBuPr2,BrBuPr}, for more recent studies.

It is worth noticing that \eqref{torsional_originial_variational} has a close connection with the variational characterization of the first eigenvalue of the Dirichlet-Laplacian and, in fact, similar techniques are used to obtain geometric estimates for both quantities (see the references above).
\vskip.2cm\noindent
We now come to the Steklov eigenvalue problem: this was originally posed by V.\,A. Steklov at the turn of the 20th century as
\begin{equation*}
\label{steklovoriginal}
	\left\{\begin{array}{rcll}
		-\Delta u&=&0,&\text{ in }\Omega,\\
		\langle \nabla u,\nu_\Omega\rangle &=&\sigma \, u,&\text{ on }\partial\Omega,
	\end{array}\right.
\end{equation*}
where $\nu_\Omega$ is the unit outer normal vector.  It plays an important role in the study of the so called \emph{sloshing problem}, in which the eigenvalues and eigenvectors of a mixed Steklov problem correspond to the fundamental frequencies and modes of vibration, respectively, of the surface of an inviscid, incompressible and heavy liquid in a container (see \cite{kuznetsov2014legacy}). Apart from this interpretation, the Steklov problem has a large number of applications in physics and engineering. For instance, it describes stationary states of a heat distribution on a bounded domain $\Omega$, where the heat flow at the boundary $\Omega$ depends on the temperature at this boundary.

Under a purely mathematical point of view, the Steklov problem has attracted a lot of interest, in particular, in the framework of \emph{spectral shape optimization}. In fact, the connection between the Steklov eigenvalues and the geometry of the domain presents some peculiarities that can not be observed in other eigenvalue problems.
In 1954, R. Weinstock showed in \cite{Wei} that, among all simply connected plane domains with given perimeter, the disc maximizes the first non-zero eigenvalue. Since then, many efforts have been devoted in order to give geometric estimates for the spectrum of the Steklov-Lapalcian: without any attempt of completenes, we cite for example \cite{BDR, brock2001isoperimetric,Bucur-Ferone-Nitsch-Trombetti,escobar1999isoperimetric,escobar2000comparison,fraser2011first,girouard2014spectral,payne1956new}.
\par
From a different point of view, the Steklov eigenvalues can be seen as the eigenvalues of the \textit{Dirichlet-to-Neumann operator}
\[
\begin{array}{ccccc}
\mathcal{D}&:&W^{\frac{1}{2},2}(\partial\Omega)&\to &W^{-\frac{1}{2},2}(\partial\Omega), \\
&& f&\mapsto & \langle \nabla\mathrm{Har}(f),\nu_\Omega\rangle,\\
\end{array}
\]
where $\mathrm{Har}(f)$ is the harmonic extension of $f$ to $\Omega$. This concept is important in many fields as electric impedance tomography, cloaking and so on (see \cite[Section 5]{kuznetsov2014legacy}).

\subsection{The boundary torsional rigidity}\label{subsection:definitions}

Let $\Omega\subset\mathbb{R}^N$ be an open bounded set with Lipschitz boundary. Given $\delta>0$, we introduce a new object, the \emph{boundary} $\delta$-\emph{torsional rigidity}, which is defined as
\begin{equation}
\label{torsionfunctional}
T(\Omega;\delta)=\sup_{\varphi\in W^{1,2}(\Omega)\setminus\{0\}}\frac{\displaystyle\left(\int_{\partial\Omega} \varphi\,d\mathcal{H}^{N-1}\right)^2}{\displaystyle\int_{\Omega}{|\nabla \varphi|^2}\,dx+ \delta^2\,\int_{\Omega}\varphi^2\,dx}.
\end{equation}
One may check (see Proposition \ref{prop:torsion} below) that the supremum is attained by any multiple of a positive function $u_{\Omega,\delta}\in W^{1,2}(\Omega)$ which is the weak solution of the boundary value problem
\begin{equation}\label{problem-introduction}
\left\{\begin{array}{rcll}
-\Delta u+\delta^2\,u&=&0,&\text{in }\Omega,\\
\langle \nabla u,\nu_\Omega\rangle &=&1,&\text{ on }\partial\Omega.
\end{array}\right.
\end{equation}
The function $u_{\Omega,\delta}$ is called the \emph{boundary $\delta-$torsion function} of $\Omega$. It is not difficult to see that we have
\begin{equation}
	\label{torsionintegral}
T(\Omega;\delta) = \int_{\partial\Omega}{u_{\Omega,\delta}}\,d\mathcal{H}^{N-1}.
\end{equation}
We note here that, while interior regularity for equation \eqref{problem-introduction} is not an issue, regularity up to the boundary is more delicate under our standing assumptions on $\Omega$. Thus, the Neumann condition in \eqref{problem-introduction} has to be intended only in weak sense (we refer to Remark  \ref{rem:smoothness}). 
\par
A strong motivation for the definition of $T(\Omega;\delta)$ comes from its relation to the ($\delta-$modified) Steklov problem. In particular, the relation to the first Steklov eigenvalue is given precisely in Proposition \ref{prop:relation-stekov}. The latter can be seen as  the Steklov version of the classical P\`olya inequality relating torsional rigidity, volume and first eigenvalue of the Dirichlet-Laplacian, see \cite[Chapter V, Section 4]{PolyaandSzego}.

Additionally, we will discuss the asymptotic behavior of both $u_{\Omega,\delta}$ and $T(\Omega;\delta)$, as the positive parameter $\delta$ goes to $0$. We will see in Theorem \ref{thm:asymptotics} that both quantities (once scaled by the dimensional factor $\delta^2$) converge to a geometric quantity, which is a kind of {\it isoperimetric--type ratio}.

\subsection{Physical interpretation}
\label{subsection:critical membrane analogy}

The \textit{membrane analogy}, also known as \textit{soap-film analogy}, was discovered by R. Prandtl in 1903. It states that the equations governing the stress distribution on a beam in torsion are the same as the ones describing the shape of a membrane deformed under a pressure applied across it (see \cite[Section 9.3.4]{sadd2009elasticity} and \cite{timoschenko1922membrane}). This fact establishes an analogy between the usual torsional rigidity and the volume under a forced membrane.

More precisely, let $\Omega\subset\mathbb{R}^2$ be an open bounded connected set. For ease of presentation, we suppose in what follows that $\Omega$ has a smooth boundary, for example of class $C^2$. Let us consider a thin membrane stretched over the rigid frame represented by $\partial\Omega$, with uniform tension $\mathcal{N}$ and under an external pressure $p$. Then, by applying the equilibrium equations to this system, we get that the vertical displacement of the membrane $v(x, y)$ at a point $(x,y)\in\Omega$ solves
\begin{equation}\label{usual-torsion}
	\left\{\begin{array}{rcll}
		-\Delta v&=&p/\mathcal{N},&\text{in }\Omega,\\
		v& =&0,&\text{on }\partial\Omega.
	\end{array}\right.
\end{equation}
When the pressure $p$ is constant, we recognize the usual torsion problem \eqref{torsion_original}, up to a multiplicative constant.
\par
Here we are interested in a Neumann version of \eqref{usual-torsion}\normalcolor:
\begin{equation*}\label{torsion-Neumann}
	\left\{\begin{array}{rcll}
		-\Delta u+\delta^2\, u&=&f(x),&\text{ in }\Omega,\\
		\langle \nabla u,\nu_\Omega\rangle&=&0,&\text{ on }\partial\Omega.
	\end{array}\right.
\end{equation*}
In this case, the boundary of the membrane can move in the vertical direction but always keeps a horizontal angle. Moreover,  the applied pressure depends on both the displacement of the membrane and the $x$ position.

For every $\varepsilon>0$ small enough, let us define the $\varepsilon-$neighbourhood of the boundary
\[
\Omega_{\varepsilon}= \Big\{x\in\Omega\, :\, d_\Omega(x)\le \varepsilon\Big\},
\]
and add a concentration effect near $\partial\Omega$ to the problem \eqref{torsion-Neumann}, that is,
\begin{equation}\label{torsionalproblemmembrane}
	\left\{\begin{array}{rcll}
		-\Delta u+\delta^2\, u&=&\dfrac{C}{\varepsilon}\,\chi_{\Omega_\varepsilon},&\text{ in }\Omega,\\
		&&&\\
		\langle \nabla u,\nu_\Omega\rangle&=&0,&\text{ on }\partial\Omega.
	\end{array}\right.
\end{equation}
The precise limit behavior of this problem is proven in  \cite[Theorem 4.1]{AJR}.
Indeed, this shows that\footnote{With the notation of \cite{AJR}, our equation \eqref{torsionalproblemmembrane} corresponds to \cite[formula (1.4)]{AJR} with the choices
\[
a(x)\equiv 1,\quad c(x)=b(x)=V_\varepsilon(x)=j_\varepsilon(x)=g_\varepsilon\equiv 0,\quad \lambda=\delta^2,\quad f_\varepsilon(x)\equiv C,
\]
and $\Gamma=\partial\Omega$.
} by taking the limit as $\varepsilon$ goes to $0$, the unique solution $u_\varepsilon$ of \eqref{torsionalproblemmembrane} converges to the solution of
\[
\left\{\begin{array}{rcll}
	-\Delta u+\delta^2\, u&=&0,&\text{ in }\Omega,\\
	\langle \nabla u,\nu_\Omega\rangle&=&C,&\text{ on }\partial\Omega,
\end{array}\right.
\]
which is our problem \eqref{problem-introduction}.
\vskip.2cm\noindent

Another possible physical motivation is the following one. We consider the time evolution of the temperature $u$ of a uniform heat conductor $\Omega$, with an initial temperature equal to $0$ and subject to a constant heat flux at the boundary. For simplicity, we normalize such a constant flux to be $1$. In other words, the function $u$ solves
\[
\left\{\begin{array}{rcll}
\Delta u&=& u_t,& \mbox{ in }\Omega\times (0,+\infty),\\
u&=&0,& \mbox{ in } \Omega\times\{0\},\\
\langle \nabla u,\nu_\Omega\rangle&=&1,& \mbox{ on }\partial\Omega\in(0,+\infty).
\end{array}
\right.
\]
If we introduce the family of probability measures $\{p_\delta\}_{\delta>0}$ on the time interval $(0,+\infty)$, given by
\[
dp_\delta(t)=\delta^2\,\exp(-\delta^2\,t)\,dt,
\]
it is not difficult to see that for every $\delta>0$ the averaged (in time) temperature
\[
U_\delta(x):=
\int_0^{+\infty} u(x,t)\,dp_\delta(t)=\delta^2\,\int_0^{+\infty} u(x,t)\,e^{-\delta^2\,t}\,dt,
\]
exactly solves \eqref{problem-introduction}, at least formally.
\par
In this way, we can think $u_{\Omega,\delta}$ as a sort of averaged temperature of the heat conductor $\Omega$; correspondingly, the quantity $T(\Omega;\delta)$ can be seen as an {\it averaged heat content} of the boundary, thanks to \eqref{torsionintegral}. This interpretation is quite close in spirit to the probabilistic description of the usual torsional rigidity, in terms of Brownian motion. We refer for example to \cite{banuelos2002torsional, PS}.

\subsection{Main results: geometric estimates}

The exact expressions of $T(\Omega;\delta)$ and $u_{\Omega,\delta}$ can be computed explicitly only in a few special geometries (balls, spherical shells or the hyperrectangles, see Section \ref{sec:4}). Thus, for a general set $\Omega$, it is important to find (possibly sharp) bounds in terms of known geometric objects. This is the main objective of this paper.
\par
For open simply connected sets in two dimensions, it is possible to use the {\it conformal transplantation technique}: this amounts to use the Riemann mapping theorem to conformally transplant test functions form the disk to any bounded, simply connected $\Omega\subset \mathbb R^2$. This is a technique already fruitfully exploited in
the literature, in order to give geometric estimates for eigenvalues of planar sets, see for example  \cite{Po55, PolyaandSzego, Sze} and \cite{Wei}. The key point here is that the Dirichlet energy is a conformal invariant. By carefully estimating the other norms contained in the definition of $T(\Omega;\delta)$, in Theorem \ref{thm:lower-bound-plane} we are able to
 give a sharp lower bound for $T(\Omega;\delta)$ in terms of the corresponding quantity for the disk and what we call {\it boundary distortion radius} of $\Omega$, defined precisely in Definition \ref{def:boundary-distortion}.
\par
On the contrary, in higher dimensions conformal mappings are not available in general and we need to modify our approach. We restrict in this case our study to the class of convex sets. First, in Theorem \ref{thm-lower-bound}, we prove a sharp lower bound on $T(\Omega;\delta)$ using the {\it method of interior parallels} introduced by Makai \cite{makai1959bounds,Mak} and P\`olya \cite{polya1960two}. We were inspired by an estimate of  P\'olya, which provides a lower bound on the usual torsional rigidity in terms of volume and perimeter, see \cite{polya1960two}. We point out that the convexity assumption could be replaced by the requirement that the distance function $d_\Omega$ to $\partial\Omega$ is weakly superharmonic, which in general is a weaker property, see Remark \ref{remark:weakly-superharmonic} below.
\par
Finally, a sharp upper bound for $T(\Omega;\delta)$ in the convex case is obtained in Theorem \ref{thm:upper-bound-convex}: this involves geometric quantities such as the inradius and the {\it proximal radius} of $\Omega$. A nice by-product of our arguments is the sharp estimate of Corollary \ref{cor:geometric}, which involves four different geometric quantities: the volume of $\Omega$, the $(N-1)-$dimensional measure of its boundary, its proximal radius and its inradius.
\subsection{Plan of the paper} In Section \ref{section:definition} we set the main notations and discuss the first properties of the boundary torsional rigidity. We discuss well-posedness of the problem and present various equivalent variational characterizations, in the vein of the classical case. We also obtain some basic upper and lower bounds on $T(\Omega;\delta)$. Then, in Section \ref{sec:3} we discuss some properties of the boundary $\delta-$torsion function $u_{\Omega,\delta}$ of a set and its asymptotic behavior as $\delta$ goes to $0$. In Section \ref{sec:4} we compute the exact shape of the boundary $\delta-$torsion function, for some special sets: the ball, the spherical shell and the hyperrectangle. Section \ref{sec:5} is entirely devoted to prove sharp geometric estimates on the boundary $\delta-$torsional rigidity. Finally, the Appendix contains some technical results on the proximal radius of a convex set.

\begin{ack}
L. Brasco gratefully acknowledges the financial support of the project FAR 2019 of the University of Ferrara.
\par
M. Gonz\'alez  acknowledges financial support from the Spanish  Government: MTM2017-85757-P, PID2020-113596GB-I00, RED2018-102650-T funded by MCIN/AEI/ 10.13039/ 501100011033, and
 the ``Severo Ochoa Programme for Centers of Excellence in R\&D'' (CEX2019-000904-S).
\par
M. Ispizua is supported by the Spanish  Government grants MTM2017-85757-P and PID2020-113596GB-I00.
\end{ack}

\section{Boundary torsional rigidity}
\label{section:definition}

\subsection{Notation}

In what follows, we will always consider $N\ge 2$. We will indicate by $B_R(x_0)$ the $N-$dimensional open ball with radius $R>0$ and center $x_0\in\mathbb{R}^N$. When the center coincides with the origin, we will simply write $B_R$. Finally, we will denote by $B$ the ball with unit radius, centered ad the origin. We also set
\[
|B|=\omega_N.
\]
For $a\in\mathbb{R}$, we will use the notation
\[
a_+=\max\{a,\, 0\}.
\]
We will use the symbol $\mathcal{H}^{N-1}$ to indicate the $(N-1)-$dimensional Hausdorff measure.
\vskip.2cm\noindent
 Let $\Omega\subset \mathbb{R}^N$ be a open bounded set, with Lipschitz boundary. We denote by $\nu_{\Omega}$ its unit outer normal vector, which is well-defined $\mathcal{H}^{N-1}-$almost everywhere on $\partial\Omega$. By $d_\Omega$ we will indicate the {\it distance function} given by
\[
d_{\Omega}(x)=\min_{y\in\partial\Omega}\left|x-y\right|,\qquad \mbox{ for } x\in\Omega.
\]
It is well-known that under the standing assumptions on $\Omega$, we have the continuous trace embedding
\[
W^{1,2}(\Omega)\hookrightarrow L^q(\partial\Omega),
\]
for every $1\le q\le 2^\#$, where
\[
2^\#=\left\{\begin{array}{ll}
\dfrac{2\,N-2}{N-2},& \mbox{ if }N\ge 3,\\
&\\
\mbox{ any finite exponent}, & \mbox{ if } N=2,
\end{array}
\right.
\]
see for example \cite[Theorems 6.4.1 \& 6.4.2]{KJF}. Moreover, such an embedding is compact, whenever $1\le q<2^\#$, see for example \cite[Remark 6.10.5, points (i) \& (ii)]{KJF}. For every admissible $q$, we will set
\begin{equation}
\label{boundary_constant}
\eta_q(\Omega)=\inf_{\varphi\in W^{1,2}(\Omega)}\Big\{\|\varphi\|_{W^{1,2}(\Omega)}^2\, :\, \|\varphi\|_{L^q(\partial\Omega)}=1\Big\}>0,
\end{equation}
which is the sharp constant for such a trace embedding.

\subsection{First properties}

Given $\delta>0$, we have defined the boundary torsional rigidity by \eqref{torsionfunctional}.
It is not difficult to see that for every $t>0$ we have the scaling law
\begin{equation}
\label{scaling}
T\left(t\,\Omega;\frac{\delta}{t}\right)=t^N\,T(\Omega;\delta).
\end{equation}
It is also easy to check that $T(\Omega,\delta)$ is related to the best Sobolev--type constant in  \eqref{boundary_constant}. More precisely:
\begin{lem}
\label{lem:basics}
Under the previous assumptions, we have $T(\Omega;\delta)<+\infty$
and the supremum in \eqref{torsionfunctional} is attained.
Moreover, the following estimates hold
\begin{equation}
\label{upperbase}
\frac{1}{\delta^2}\,\frac{(\mathcal{H}^{N-1}(\partial\Omega))^2}{|\Omega|}\le T(\Omega;\delta)\le \frac{1}{\min\{1,\delta^2\}}\,\frac{1}{\eta_1(\Omega)},
\end{equation}
where $\eta_1(\Omega)$ is defined in \eqref{boundary_constant}.
\end{lem}

\begin{proof}
We first observe that
\[
T(\Omega;\delta)\le \frac{1}{\min\{1,\delta^2\}}\, \sup_{\varphi\in W^{1,2}(\Omega)\setminus\{0\}}\frac{\displaystyle\left(\int_{\partial\Omega} \varphi\,d\mathcal{H}^{N-1}\right)^2}{\displaystyle\int_{\Omega}{|\nabla \varphi|^2}\,dx+ \int_{\Omega}\varphi^2\,dx}=\frac{1}{\min\{1,\delta^2\}}\,\frac{1}{\eta_1(\Omega)}.
\]
This shows that $T(\Omega;\delta)<+\infty$ and also proves the upper bound in \eqref{upperbase}.
\vskip.2cm\noindent
Existence of an extremal for $T(\Omega;\delta)$ can be proved by appealing to the Direct Method in the Calculus of Variations. It is sufficient to observe that the maximization in \eqref{torsionfunctional} is equivalently performed on the set $W^{1,2}(\Omega)\setminus W^{1,2}_0(\Omega)$, then a maximizer of \eqref{torsionfunctional} can be obtained by solving
\begin{equation}
\label{constrained}
\inf_{u\in W^{1,2}(\Omega)} \left\{\int_{\Omega}{|\nabla \varphi|^2}\,dx+ \delta^2\,\int_{\Omega}\varphi^2\,dx\, :\, \int_{\partial\Omega} \varphi\,d\mathcal{H}^{N-1}=1\right\}.
\end{equation}
Any minimizing sequence $\{\varphi_n\}_{n\in\mathbb{N}}\subset W^{1,2}(\Omega)$ for this problem is such that
\[
\int_{\Omega}{|\nabla \varphi_n|^2}\,dx+\int_{\Omega}\varphi^2_n\,dx\le \frac{C}{
\min\{1,\delta^2\}}\qquad \mbox{ and } \qquad  \int_{\partial\Omega} \varphi_n\,d\mathcal{H}^{N-1}=1,
\]
for some $C>0$ independent of $n$.
\par
The first property, in conjuction with the Rellich-Kondra\v{s}ov Theorem and the compactness of the trace embedding $W^{1,2}(\Omega)\hookrightarrow L^1(\partial\Omega)$, implies that there exists $u_0\in W^{1,2}(\Omega)$ such that $\varphi_n$ converges to $u_0$ (up to a subsequence), weakly in $W^{1,2}(\Omega)$, strongly in $L^2(\Omega)$ and strongly in $L^1(\partial\Omega)$. In particular, $u_0$ has still unit boundary integral.
Thus the limit function $u_0$ is still admissible in \eqref{constrained} and, thanks to the lower semicontinuity of the functional, we get that $u_0$ is the desired solution.
\vskip.2cm\noindent
Finally, in order to prove the lower bound \eqref{upperbase}, it is sufficient to use the characteristic function of $\Omega$ as a test function in the Rayleigh--type quotient which defines $T(\Omega;\delta)$.
\end{proof}
\begin{rem}
As simple as it may appear, the lower bound in \eqref{upperbase} is actually sharp. This will be shown in Remark \ref{rem:sharpbase}, as a consequence of the asymptotic behavior of the quantity
\[
\delta^2\,T(\Omega;\delta),\qquad \mbox{ as } \delta\searrow 0,
\]
see Theorem \ref{thm:asymptotics} below.
For the moment, we observe that from \eqref{upperbase} we get that
\[
\lim_{\delta\to 0^+} T(\Omega;\delta)=+\infty,
\]
and $T(\Omega;\delta)$ diverges precisely at a rate $\delta^{-2}$.
\end{rem}
As for the usual torsional rigidity, $T(\Omega;\delta)$ can be equivalently defined through an unconstrained concave maximization problem:
\begin{prop}
\label{prop:torsion}
Let $\delta>0$ and let $\Omega\subset\mathbb{R}^N$ be an open bounded set, with Lipschitz boundary. Then we have
\begin{equation}
\label{unconstrained-min}
T(\Omega;\delta)=\sup_{\varphi\in W^{1,2}(\Omega)}\left\{2\,\int_{\partial\Omega} \varphi\,d\mathcal{H}^{N-1} -\int_{\Omega}|\nabla \varphi|^2\,dx-\delta^2\,\int_{\Omega}\varphi^2\,dx\right\}.
\end{equation}
Moreover, the supremum in \eqref{unconstrained-min} is uniquely attained by a non-negative function $u_{\Omega,\delta}\in W^{1,2}(\Omega)$, which is the weak solution of the Neumann boundary value problem
\[
\left\{\begin{array}{rcll}
-\Delta u+\delta^2\,u&=&0,&\text{in }\Omega,\\
\langle \nabla u,\nu_\Omega\rangle &=&1,&\text{on }\partial\Omega.
\end{array}\right.
\]
In other words, $u_{\Omega,\delta}$ satisfies
\begin{equation}
\label{torsionalproblem}
\int_\Omega \langle \nabla u_{\Omega,\delta},\nabla \varphi\rangle\,dx+\delta^2\,\int_\Omega u_{\Omega,\delta}\,\varphi\,dx=\int_{\partial\Omega} \varphi\,d\mathcal{H}^{N-1},\quad \mbox{ for every } \varphi\in W^{1,2}(\Omega).
\end{equation}
Finally, we also have \eqref{torsionintegral}, this is
\begin{equation*}
T(\Omega;\delta) = \int_{\partial\Omega}{u_{\Omega,\delta}}\,d\mathcal{H}^{N-1}.
\end{equation*}
\end{prop}
\begin{proof}
The existence of a solution for the maximization problem in \eqref{unconstrained-min} follows again from the Direct Method in the Calculus of Variations.
Indeed, we first observe that by using the definition \eqref{boundary_constant} of $\eta_1(\Omega)$ and Young's inequality we get, for every $L>0$,
\[
\begin{split}
\mathfrak{F}(\varphi):=2\,\int_{\partial\Omega} \varphi\,d\mathcal{H}^{N-1}& -\int_{\Omega}|\nabla \varphi|^2\,dx-\delta^2\,\int_{\Omega}\varphi^2\,dx\\
&\le 2\,\int_{\partial\Omega} \varphi\,d\mathcal{H}^{N-1} -\min\{1,\delta^2\}\,\|\varphi\|^2_{W^{1,2}(\Omega)}\\
&\le \frac{2}{\sqrt{\eta_1(\Omega)}}\,\|\varphi\|_{W^{1,2}(\Omega)}-\min\{1,\delta^2\}\,\|\varphi\|^2_{W^{1,2}(\Omega)}\\
&\le \Big(L-\min\{1,\delta^2\}\Big)\,\|\varphi\|_{W^{1,2}(\Omega)}^2+\frac{1}{L\,\eta_1(\Omega)}.
\end{split}
\]
By choosing $L=\min\{1,\delta^2\}/2$, we  obtain
\begin{equation}
\label{coercive}
\mathfrak{F}(\varphi)\le -\frac{1}{C}\,\|\varphi\|^2_{W^{1,2}(\Omega)}+C,
\end{equation}
for some $C=C(N,\Omega,\delta)>0$. This shows that
\begin{equation*}
\sup_{\varphi\in W^{1,2}(\Omega)} \mathfrak{F}(\varphi)<+\infty.
\end{equation*}
Consequently, by the estimate \eqref{coercive}, every maximizing sequence $\{\varphi_n\}_{n\in\mathbb{N}}$ is equi-bounded in $W^{1,2}(\Omega)$. Existence of a maximizer $u_{\Omega,\delta}$ can now be inferred as in the proof of Lemma \ref{lem:basics}.
\par
Uniqueness of the solution follows from the strict concavity of the functional, while the definite sign property is a consequence of the uniqueness and the fact that
\[
\mathfrak{F}(|\varphi|)\ge \mathfrak{F}(\varphi),\qquad \mbox{ for every } \varphi\in W^{1,2}(\Omega).
\]
 Finally, the equation \eqref{torsionalproblem} is just the optimality condition for $u_{\Omega,\delta}$, which can be obtained by computing the first variation of the functional.
\vskip.2cm\noindent
We now come to the proof of \eqref{unconstrained-min}. We first notice that
\[
\begin{split}
\max_{\varphi\in W^{1,2}(\Omega)}&\left\{2\,\int_{\partial\Omega} \varphi\,d\mathcal{H}^{N-1} -\int_{\Omega}|\nabla \varphi|^2\,dx-\delta^2\,\int_{\Omega}\varphi^2\,dx\right\}\\
&=\max_{\varphi\in W^{1,2}(\Omega)\setminus W^{1,2}_0(\Omega)}\left\{2\,\int_{\partial\Omega} |\varphi|\,d\mathcal{H}^{N-1} -\int_{\Omega}|\nabla \varphi|^2\,dx-\delta^2\,\int_{\Omega}\varphi^2\,dx\right\}.
\end{split}
\]
Indeed, observe that for every $\varphi\in W^{1,2}_0(\Omega)$ we have
\[
2\,\int_{\partial\Omega} |\varphi|\,d\mathcal{H}^{N-1} -\int_{\Omega}|\nabla \varphi|^2\,dx-\delta^2\,\int_{\Omega}\varphi^2\,dx\le 0,
\]
while, if $\varphi\in W^{1,2}(\Omega)\setminus W^{1,2}_0(\Omega)$, then for every $t>0$ small enough we have
\[
2\,t\,\int_{\partial\Omega} |\varphi|\,d\mathcal{H}^{N-1} -t^2\,\left(\int_{\Omega}|\nabla \varphi|^2\,dx+\delta^2\,\int_{\Omega}\varphi^2\,dx\right)>0.
\]
This also gives in particular that
\begin{equation}
	\label{tminimization}
	\begin{split}
		&\max_{\varphi\in W^{1,2}(\Omega)\setminus W^{1,2}_0(\Omega)}\left\{2\,\int_{\partial\Omega} |\varphi|\,d\mathcal{H}^{N-1} -\int_{\Omega}|\nabla \varphi|^2\,dx-\delta^2\,\int_{\Omega}\varphi^2\,dx\right\}\\
		&=\max_{\varphi\in W^{1,2}(\Omega)\setminus W^{1,2}_0(\Omega)}\sup_{t>0}\left\{2\,t\,\int_{\partial\Omega} |\varphi|\,d\mathcal{H}^{N-1}-t^2\,\left(\int_{\Omega}|\nabla \varphi|^2\,dx-\delta^2\,\int_{\Omega}\varphi^2\,dx\right)\right\}.
	\end{split}
\end{equation}
For every $\varphi\in W^{1,2}(\Omega)\setminus W^{1,2}_0(\Omega)$, the optimal choice of $t$ is given by
\[
t=\frac{\displaystyle\int_{\partial\Omega}{\varphi}\,d\mathcal{H}^{N-1}}{\displaystyle\int_{\Omega}{|\nabla \varphi|^2\,dx}+ \delta^2\,\int_{\Omega} \varphi^2\,dx}.
\]
By substituting, we get \eqref{unconstrained-min}.
\vskip.2cm\noindent
Let us finally prove formula \eqref{torsionintegral}.
From \eqref{unconstrained-min}  it follows
\[
T(\Omega;\delta)=2\,\int_{\partial\Omega} u_{\Omega,\delta}\,d\mathcal{H}^{N-1} -\int_{\Omega}|\nabla u_{\Omega,\delta}|^2\,dx-\delta^2\,\int_{\Omega}u_{\Omega,\delta}^2\,dx.
\]
On the other hand, by taking $\varphi=u_{\Omega,\delta}$ as a test function in \eqref{torsionalproblem}, we obtain
\[
\int_{\Omega}|\nabla u_{\Omega,\delta}|^2\,dx+\delta^2\,\int_{\Omega}u_{\Omega,\delta}^2\,dx=\int_\Omega u_{\Omega,\delta}\,dx.
\]
By joining the last two equations in display, we obtain \eqref{torsionintegral}.
\end{proof}

\begin{defn}
The function $u_{\Omega,\delta}$ of Proposition \ref{prop:torsion} will be called {\it boundary $\delta-$torsion function of $\Omega$}. We also observe that such a function enjoys the following scaling law for $t>0$:
\begin{equation}
\label{scaling_u}
u_{t\,\Omega,\delta/t}(x)=t\,u_{\Omega,\delta}\left(\frac{x}{t}\right),\qquad \mbox{ for every } x\in t\,\Omega.
\end{equation}
\end{defn}
It is not difficult to see that the concave maximization problem in \eqref{unconstrained-min} admits a dual formulation, in the sense of Convex Analysis. This in turn permits to equivalently define $T(\Omega;\delta)$ as a convex minimization problem. This fact  will be useful in the sequel.
\begin{lem}[Dual formulation]\label{lm:dual_formulation}
Let $\Omega\subset\mathbb{R}^N$ be an open bounded set, with Lipschitz boundary. Let us set
\[
\mathcal{A}^{+}(\Omega)=\left\{(\phi, g)\in L^2(\Omega;\mathbb{R}^N)\times L^2(\Omega)\, :\, \begin{array}{rc}-\mathrm{div\,}\phi+\delta^2\, g\geq0, &\mbox{in }\Omega\\ \langle\phi,\nu_{\Omega}\rangle \geq 1, &\mbox{on }\partial\Omega\end{array}\right\},
\]
where the conditions have to be intended in weak sense, i.e.
\[
\int_{\partial\Omega} \varphi\,d\mathcal{H}^{N-1}\le \int_\Omega \langle \phi,\nabla \varphi\rangle\,dx+\delta^2\,\int_\Omega g\,\varphi\,dx,\qquad \mbox{ for every } \varphi\in W^{1,2}(\Omega) \mbox{ non-negative}.
\]
Then we have
	\begin{equation}\label{dual functional}
		T(\Omega;\delta)=\min_{(\phi, g)\in\mathcal{A}^+(\Omega)}\left\{\int_{\Omega}|\phi|^2\,dx+\delta^2\int_{\Omega}g^2\,dx\right\},
	\end{equation}
and the minimum is uniquely attained by the pair $(\nabla u_{\Omega,\delta},u_{\Omega,\delta})$.
\end{lem}
\begin{proof}
For every non-negative $\varphi\in W^{1,2}(\Omega)$ and every $(\phi,g)\in\mathcal{A}^+(\Omega)$, we have by Young's inequality
\[
\begin{split}
\int_{\partial\Omega} \varphi\,d\mathcal{H}^{N-1}&\le \int_\Omega \langle \phi,\nabla \varphi\rangle\,dx+\delta^2\,\int_\Omega g\,\varphi\,dx\\
&\le \frac{1}{2}\,\int_\Omega |\phi|^2\,dx+\frac{1}{2}\,\int_\Omega |\nabla \varphi|^2\,dx+\frac{\delta^2}{2}\,\int_\Omega g^2\,dx+\frac{\delta^2}{2}\,\int_\Omega \varphi^2\,dx.
\end{split}
\]	
This in particular gives
\[
		2\,\int_{\partial\Omega}\varphi\,d\mathcal{H}^{N-1}-\left(\int_{\Omega}|\nabla \varphi|^2\,dx+\delta^2\,\int_{\Omega}\varphi^2\,dx\right)
\le \int_{\Omega}|\phi|^2\,dx+\delta^2\,\int_{\Omega}g^2\,dx.
\]
and by arbitrariness of $\varphi$ and $(\phi,g)$ we get
\[
		\max_{\varphi\in W^{1,2}(\Omega), \varphi\ge 0}\left\{2\int_{\partial\Omega}\varphi\,d\sigma-\int_{\Omega}|\nabla \varphi|^2\,dx-\delta^2\int_{\Omega}\varphi^2\,dx\right\}\leq\inf_{(\phi, g)\in\mathcal{A}^+(\Omega)}\left\{\int_{\Omega}|\phi|^2\,dx+\delta^2\,\int_{\Omega}g^2\,dx\right\}.
\]	
Since the constraint $\varphi\ge 0$ can be dropped without affecting the maximum value, recalling \eqref{unconstrained-min} we get
\[
T(\Omega;\delta)\le \inf_{(\phi, g)\in\mathcal{A}^+(\Omega)}\left\{\int_{\Omega}|\phi|^2\,dx+\delta^2\,\int_{\Omega}g^2\,dx\right\}.
\]
On the other hand, it is easy to see that the pair $(\phi_0,g_0)=(\nabla u_{\Omega,\delta}, u_{\Omega,\delta})$ is admissible and it gives
\[
\begin{split}
\inf_{(\phi, g)\in\mathcal{A}^+(\Omega)}\left\{\int_{\Omega}|\phi|^2\,dx+\delta^2\,\int_{\Omega}g^2\,dx\right\}&\le \int_\Omega |\nabla u_{\Omega,\delta}|^2\,dx+\delta^2\,\int_\Omega u_{\Omega,\delta}^2\,dx\\
&=\int_{\partial\Omega} u_{\Omega,\delta} \,d\mathcal{H}^{N-1}=T(\Omega;\delta).
\end{split}
\]
This finally proves \eqref{dual functional} and the fact that $(\nabla u_{\Omega,\delta},u_{\Omega,\delta})$ is a minimizer. Its uniqueness easily follows from the strict convexity of the variational problem in \eqref{dual functional}.
\end{proof}

\subsection{Relation with a Steklov eigenvalue problem}

Here we establish some relations between the boundary torsional rigidity and the first Steklov eigenvalue of the Schr\"odinger--type operator
\[
u\mapsto -\Delta u+\delta^2\,u.
\]
The Steklov spectrum of this operator is made of the real numbers $\sigma$ such that the following boundary value problem
\[
	\left\{\begin{array}{rcll}
		-\Delta u+\delta^2\,u&=&0,&\text{in }\Omega,\\
		\langle \nabla u,\nu_{\Omega}\rangle& =&\sigma\, u,& \text{on }\partial\Omega,
	\end{array}\right.
\]
admits at least a non-trivial weak solution $u\in W^{1,2}(\Omega)$.
It is well known that the whole spectrum is made of an increasing sequence of eigenvalues $\{\sigma_n(\Omega;\delta)\}_{n\in\mathbb{N}}$ diverging at infinity. Indeed, it is sufficient to observe that the resolvent operator
\[
\begin{array}{ccccc}
\mathcal{R}_\delta&:& L^2(\partial\Omega)& \to& L^2(\partial\Omega)\\
&& f & \mapsto & u_f,
\end{array}
\]
is positive, self-adjoint and compact. Then we can apply the Spectral Theorem in order to prove the claimed structure of the spectrum. Here, by $u_f\in W^{1,2}(\Omega)$ we mean the unique weak solution to the Neumann boundary value problem
\[
	\left\{\begin{array}{rcll}
		-\Delta u+\delta^2\,u&=&0,&\text{in }\Omega,\\
		\langle \nabla u,\nu_{\Omega}\rangle& =&f,& \text{on }\partial\Omega.
	\end{array}\right.
\]
Of course, compactness of $\mathcal{R}_\delta$ is again due to the compactness of the relevant trace embedding
\[
W^{1,2}(\Omega) \hookrightarrow L^2(\partial\Omega).
\]
The first eigenvalue has the following variational characterization
\begin{equation*}
	\sigma_1(\Omega;\delta)=\min_{u\in W^{1,2}(\Omega)\setminus W^{1,2}_0(\Omega)}\frac{\displaystyle\int_{\Omega}|\nabla u|^2\,dx+ \delta^2\,\int_{\Omega}u^2\,dx}{\displaystyle\int_{\partial\Omega} u^2\,d\mathcal{H}^{N-1}}.
\end{equation*}
Observe that by choosing $u$ to be the characteristic function of $\Omega$, we obtain
\[
\sigma_1(\Omega;\delta)\le \delta^2 \,\frac{|\Omega|}{\mathcal{H}^{N-1}(\partial\Omega)}\qquad \mbox{ and thus }\qquad \lim_{\delta\to 0^+}\sigma_1(\Omega;\delta)=0.
\]
This is consistent with the fact that the first eigenvalue of the Steklov-Laplacian is $0$, associated to constant eigenfunctions (see for example \cite[Chapter 7, Section 3]{He}).
\vskip.2cm
The relation between $\sigma_1(\Omega;\delta)$ and $T(\Omega;\delta)$ follows immediately by  applying H\"older's inequality on the boundary integral
\begin{equation*}
\left(\int_{\partial\Omega}u\,d\mathcal{H}^{N-1}\right)^2\leq
\mathcal{H}^{N-1}(\partial\Omega)\, \int_{\partial\Omega}u^2\,d\mathcal{H}^{N-1}.
\end{equation*}
This yields the following estimate, which should be compared with the classical P\`olya inequality relating torsional rigidity, volume and first eigenvalue of the Dirichlet-Laplacian, see \cite[Chapter V, Section 4]{PolyaandSzego}:

\begin{prop}\label{prop:relation-stekov}
Let $\delta>0$ and let $\Omega\subset\mathbb{R}^N$ be an open bounded set, with Lipschitz boundary. Then we have
\begin{equation}
\frac{\sigma_1(\Omega;\delta)\,T(\Omega;\delta)}{\mathcal{H}^{N-1}(\partial\Omega)}\le 1.
\end{equation}
\end{prop}

\section{Some properties of the torsion function}
\label{sec:3}

In this section, we establish some few quantitative properties of the function $u_{\Omega,\delta}$ and study its asymptotic behavior as $\delta$ goes to $0$.
\begin{prop}
\label{prop:ubasic}
Let $\delta>0$ and let $\Omega\subset\mathbb{R}^N$ be an open bounded set, with Lipschitz boundary.
\begin{enumerate}
\item[{\it (i)}] The $L^1(\Omega)$ norm of $u_{\Omega,\delta}$ is given by
\begin{equation}
\label{L1}
\int_\Omega u_{\Omega,\delta}\,dx=\frac{\mathcal{H}^{N-1}(\partial\Omega)}{\delta^2};
\end{equation}
\item[{\it (ii)}] its trace is in $L^\infty(\partial\Omega)$, with the following estimate: for every exponent $2<q<2^\#$
\begin{equation}
\label{doublesided}
\|u_{\Omega,\delta}\|_{L^\infty(\partial\Omega)}\le C_q\,\left(\frac{T(\Omega;\delta)^\frac{q-2}{q}}{\min\{1,\delta^2\}\,\eta_{q}(\Omega)}\right)^\frac{q}{2\,(q-1)},
\end{equation}
where $\eta_q(\Omega)$ is the constant defined in \eqref{boundary_constant} and $C_q>0$ is a constant only depending on $q$, which blows-up as $q\searrow 2$;
\vskip.2cm
\item[{\it (iii)}] we have $u_{\Omega,\delta}\in L^\infty(\Omega)$ and it holds
\begin{equation}
\label{maximumprinciple}
\|u_{\Omega,\delta}\|_{L^\infty(\Omega)}\le \|u_{\Omega,\delta}\|_{L^\infty(\partial\Omega)}.
\end{equation}
\end{enumerate}
\end{prop}
\begin{proof}
Point \emph{(i)} simply follows by testing \eqref{torsionalproblem} with the characteristic function of $\Omega$.
\vskip.2cm\noindent
We now prove that $u_{\Omega,\delta}$ has a bounded trace, by using a Moser's iteration.
We first show that it is sufficient to prove the claimed estimate \eqref{L1} for $\delta=1$. Indeed, let us suppose that \eqref{L1} holds true for $\delta=1$, for every open bounded set, with Lipschitz boundary.
By recalling \eqref{scaling_u}, we would get
\[
u_{\Omega,\delta}(x)=\frac{1}{\delta}\,u_{\delta\,\Omega,1}(\delta\,x),\qquad \mbox{ for }x\in\Omega,
\]
and thus
\begin{equation}
\label{reduction}
\|u_{\Omega,\delta}\|_{L^\infty(\partial\Omega)}=\frac{1}{\delta}\,\|u_{\delta\,\Omega,1}\|_{L^\infty(\partial(\delta\,\Omega))}\le \frac{C_q}{\delta}\,\left(\frac{T(\delta\,\Omega;1)^\frac{q-2}{q}}{\eta_{q}(\delta\,\Omega)}\right)^\frac{q}{2\,(q-1)}.
\end{equation}
We now observe that by its definition and a change of variable, we have
\[
\begin{split}
\eta_{q}(\delta\,\Omega)&=\inf_{\varphi\in W^{1,2}(\delta\Omega)\setminus W^{1,2}_0(\delta\Omega)}\frac{\displaystyle\int_{\delta\Omega}{|\nabla \varphi|^2}\,dx+ \int_{\delta\Omega}\varphi^2\,dx}{\displaystyle\left(\int_{\partial\Omega} |\varphi|^{q}\,d\mathcal{H}^{N-1}\right)^\frac{2}{q}}\\
&=\inf_{\varphi\in W^{1,2}(\Omega)\setminus W^{1,2}_0(\Omega)}\frac{\delta^{N-2}\,\displaystyle\int_{\Omega}{|\nabla \varphi|^2}\,dx+ \delta^N\,\int_{\Omega}\varphi^2\,dx}{\delta^{(N-1)\frac{2}{q}}\displaystyle\left(\int_{\partial\Omega} |\varphi|^{q}\,d\mathcal{H}^{N-1}\right)^\frac{2}{q}}\\
&\ge \delta^{N\,\frac{q-2}{q}-2\,\frac{q-1}{q}}\,\Big(\min\{1,\delta^2\}\Big)\,\eta_{q}(\Omega),
\end{split}
\]
while by \eqref{scaling}
\[
T(\delta\,\Omega;1)=\delta^N\,T(\Omega;\delta).
\]
Going back to \eqref{reduction}, we get the claimed upper bound in \eqref{doublesided} for $\delta\not=1$, as well.\\

In order to prove \eqref{doublesided} with $\delta=1$, for notational simplicity we will write $u$ in place of $u_{\Omega,1}$. We fix $M>0$ and $\beta\ge 1$, then we insert the test function
\[
\varphi=u_M^\beta,\qquad \mbox{ where } u_M=\min\{u,M\},
\]
in the weak formulation \eqref{torsionalproblem}. Observe that this is a feasible test function, thanks to the Chain Rule in Sobolev spaces. We then obtain
\[
\beta\,\int_{\{u\le M\}} |\nabla u|^2\,u^{\beta-1}\,dx+\,\int_\Omega u\,u_M^\beta\,dx=\int_{\partial\Omega} u_M^\beta\,d\mathcal{H}^{N-1}.
\]
By observing that
\[
\beta\,\int_{\{u\le M\}} |\nabla u|^2\,u^{\beta-1}\,dx=\frac{4\,\beta}{(\beta+1)^2}\,\int_\Omega \left|\nabla \left(u_M^\frac{\beta+1}{2}\right)\right|^2\,dx,
\]
and using that $u\ge u_M$,
from the identity above we get
\[
\frac{4\,\beta}{(\beta+1)^2}\,\int_\Omega \left|\nabla \left(u_M^\frac{\beta+1}{2}\right)\right|^2\,dx+\int_\Omega \left(u_M^\frac{\beta+1}{2}\right)^2\,dx\le \int_{\partial\Omega} u_M^\beta\,d\mathcal{H}^{N-1}.
\]
This in particular entails that
\[
\left\|u_M^\frac{\beta+1}{2}\right\|_{W^{1,2}(\Omega)}^2\le \left(\frac{(\beta+1)^2}{4\,\beta}+1\right)\,\int_{\partial\Omega} u_M^\beta\,d\mathcal{H}^{N-1}.
\]
 By recalling the notation \eqref{boundary_constant}, from the trace embedding with $2<q<2^\#$ we obtain
\begin{equation}
\label{moser0}
\eta_{q}(\Omega)\,\left(\int_{\partial\Omega} \left(u_M^\frac{\beta+1}{2}\right)^{q}\,d\mathcal{H}^{N-1}\right)^\frac{2}{q} \le \left(\frac{(\beta+1)^2}{4\,\beta}+1\right)\,\int_{\partial\Omega} u_M^\beta\,d\mathcal{H}^{N-1},
\end{equation}
which is an iterative scheme of reverse H\"older inequalities. Before going further, we observe that for every $\beta\ge 1$ we have
\[
\frac{(\beta+1)^2}{4\,\beta}+1\le \frac{\beta+1}{2}+1\le \beta+1\le 2\,\beta,
\]
and then define the sequence of exponents
\[
\beta_0=1,\qquad \beta_{i+1}=\frac{q}{2}\,(\beta_i+1),\qquad \mbox{ for } i\in\mathbb{N}.
\]
Then from \eqref{moser0} we get
\[
\left(\int_{\partial\Omega} (u_M)^{\beta_{i+1}}\,d\mathcal{H}^{N-1}\right)^\frac{1}{\beta_{i+1}} \le \left(\frac{2}{\eta_{q}(\Omega)}\,\beta_i\right)^{\frac{q}{2}\,\frac{1}{\beta_{i+1}}}\,\left(\int_{\partial\Omega} u_M^{\beta_i}\,d\mathcal{H}^{N-1}\right)^{\frac{q}{2}\,\frac{1}{\beta_{i+1}}}.
\]
If we set for notational simplicity $Y_i=\|u_M\|_{L^{\beta_i}(\partial\Omega)}$, the previous scheme can be written as
\begin{equation}
\label{moser1}
Y_{i+1}\le  \left(\frac{2}{\eta_{q}(\Omega)}\,\beta_i\right)^{\frac{q}{2}\,\frac{1}{\beta_{i+1}}}\, Y_i^{\frac{q}{2}\,\frac{\beta_i}{\beta_{i+1}}},\qquad \mbox{ for every } i\in\mathbb{N}.
\end{equation}
We start with $i=0$ and iterate \eqref{moser1} $n$ times. We then obtain
\begin{equation}
\label{moser2}
Y_{n}\le \left(\frac{2}{\eta_{q}(\Omega)}\right)^{\frac{1}{\beta_n}\,\sum\limits_{i=1}^{n} \left(\frac{q}{2}\right)^i}\,\left[\prod_{i=0}^{n-1}\beta_i^{\left(\frac{q}{2}\right)^{n-i}}\right]^\frac{1}{\beta_n}\, Y_0^{\left(\frac{q}{2}\right)^n\,\frac{\beta_0}{\beta_{n}}}.
\end{equation}
We now wish to take the limit as $n$ goes to $\infty$ in this estimate. For this, we observe that
\[
\sum\limits_{i=1}^{n} \left(\frac{q}{2}\right)^i\sim \frac{2}{q-2}\,\left(\frac{q}{2}\right)^{n+1},\qquad \mbox{ as } n\to\infty,
\]
\[
\beta_n=\left(\left(\frac{q}{2}\right)^n+\sum_{i=1}^n \left(\frac{q}{2}\right)^i\right)\sim \left(\frac{q}{2}\right)^n\,\frac{2\,(q-1)}{q-2},\qquad \mbox{ as } n\to\infty,
\]
and
\[
\begin{split}
\lim_{n\to\infty} \left[\prod_{i=0}^{n-1}\beta_i^{\left(\frac{q}{2}\right)^{n-i}}\right]^\frac{1}{\beta_n}&=\lim_{n\to\infty} \exp\left(\frac{1}{\beta_n}\,\left(\frac{q}{2}\right)^n\,\sum_{i=0}^{n-1}\left(\frac{2}{q}\right)^i\,\log\beta_i\right)=C_q<+\infty.
\end{split}
\]
Thus from \eqref{moser2}, we get
\[
Y_\infty\le C_q\,\left(\frac{2}{\eta_{q}(\Omega)}\right)^\frac{q}{2\,(q-1)}\,\Big(Y_0\Big)^\frac{q-2}{2\,(q-1)}.
\]
Finally, by recalling the definition of $Y_i$ and that of $u_M$, we get
\[
\begin{split}
\|u_M\|_{L^\infty(\partial\Omega)}&\le  C_q\,\left(\frac{2}{\eta_{q}(\Omega)}\right)^\frac{q}{2\,(q-1)}\,\Big(\int_{\partial\Omega} u_M\,d\mathcal{H}^{N-1}\Big)^\frac{q-2}{2\,(q-1)}\\
&\le C_q\,\left(\frac{2}{\eta_{q}(\Omega)}\right)^\frac{q}{2\,(q-1)}\,\Big(\int_{\partial\Omega} u\,d\mathcal{H}^{N-1}\Big)^\frac{q-2}{2\,(q-1)}=C_q\,\left(\frac{2}{\eta_{q}(\Omega)}\right)^\frac{q}{2\,(q-1)}\,\Big(T(\Omega;\delta)\Big)^\frac{q-2}{2\,(q-1)}.
\end{split}
\]
We also used the identity \eqref{torsionintegral} in the last equality.
We eventually get the desired result from the previous estimate, by arbitrariness of $M>0$. This concludes the proof of point \emph{(ii)}.
\vskip.2cm\noindent
Finally, we come to the proof of point \emph{(iii)}, i.e. the maximum principle \eqref{maximumprinciple}. Let us set for brevity $L=\|u_{\Omega,\delta}\|_{L^\infty(\partial\Omega)}$, then we have
\[
u_{\Omega,\delta}\le L,\qquad \mbox{ on } \partial\Omega.
\]
This implies that $\varphi=((u_{\Omega,\delta})-L)_+\in W^{1,2}_0(\Omega)$. By inserting this test function in \eqref{torsionalproblem}, we obtain
\[
\int_\Omega \langle \nabla u_{\Omega,\delta},\nabla ((u_{\Omega,\delta})-L)_+\rangle\,dx+\delta^2\,\int_\Omega u_{\Omega,\delta}\,((u_{\Omega,\delta})-L)_+\,dx=0,
\]
that is
\[
\int_\Omega |\nabla ((u_{\Omega,\delta})-L)_+|^2\,dx+\delta^2\,\int_\Omega u_{\Omega,\delta}\,((u_{\Omega,\delta})-L)_+\,dx=0.
\]
Since both terms are non-negative, we get that they both must vanish. This in turn implies
\[
\nabla ((u_{\Omega,\delta})-L)_+=0,\qquad \mbox{ a.\,e. in }\Omega,
\]
and
\[
u_{\Omega,\delta}\,((u_{\Omega,\delta})-L)_+=0,\qquad \mbox{ a.\,e. in }\Omega.
\]
The two informations entail that $(u_{\Omega,\delta}-L)_+$ must vanish almost everywhere in $\Omega$, which means that
\[
0\le u_{\Omega,\delta}\le L= \|u_{\Omega,\delta}\|_{L^\infty(\partial\Omega)},\qquad \mbox{ a.\,e. in }\Omega.
\]
Thus we get $\|u_{\Omega,\delta}\|_{L^\infty(\Omega)}\le \|u_{\Omega,\delta}\|_{L^\infty(\partial\Omega)}$.
\end{proof}
\begin{rem}[Regularity of $u_{\Omega,\delta}$]
\label{rem:smoothness}
We remark that by classical Elliptic Regularity, the function $u_{\Omega,\delta}$ actually belongs to $C^\infty(\Omega)$, i.e. it is a classical solution of
\[
-\Delta u+\delta^2\,u=0,\qquad \mbox{ in }\Omega.
\]
However, the regularity up to the boundary is more delicate, under our standing assumptions on $\Omega$. Indeed, well-known counterexamples show that global smoothness may fail in Lipschitz sets, already for harmonic functions and very simple boundary data, see for example \cite{Gr}. Thus in general $u_{\Omega,\delta}$ has to be considered only as a weak solution of
\[
\left\{\begin{array}{rcll}
-\Delta u+\delta^2\,u&=&0,&\text{in }\Omega,\\
\langle \nabla u,\nu_\Omega\rangle &=&1,&\text{on }\partial\Omega,
\end{array}\right.
\]
i.e. the Neumann condition has to be intended only in weak sense.
\par
Finally, we point out that the property $u_{\Omega,\delta}\ge 0$ can actually be enforced to
\[
u_{\Omega,\delta}>0,\qquad \mbox{ for every } x\in\Omega,
\]
by virtue of the minimum principle, see \cite[Chapter 6, Section 4, Theorem 4]{Ev}. Observe that this holds true also if $\Omega$ is a disconnected set, by virtue of the non-homogeneous Neumann condition. In other words, if $\Omega$ is made of $k$ connected components, we must have $u_{\Omega,\delta}>0$ on each single component.
\end{rem}
The boundary torsion functions are monotone with respect to the parameter $\delta$. This is the content of the next result:
\begin{lem}[Monotonicity]
Let $\Omega\subset\mathbb{R}^N$ be an open bounded set, with Lipschitz boundary. For every $0<\delta_0<\delta_1$, we have
\[
u_{\Omega,\delta_0}> u_{\Omega,\delta_1},\qquad \mbox{ in }\Omega.
\]
\end{lem}
\begin{proof}
Let us set for simplicity
\[
u_i:=u_{\Omega,\delta_i},\qquad i=0,1.
\]
By subtracting the equations \eqref{torsionalproblem} for $u_0$ and $u_1$, we get
\begin{equation}
\label{subtract}
\int_\Omega \langle \nabla (u_1-u_0),\nabla \varphi\rangle\,dx+\delta_1^2\,\int_\Omega u_1\,\varphi\,dx-\delta_0^2\,\int_\Omega u_0\,\varphi\,dx=0,\qquad \mbox{ for every } \varphi\in W^{1,2}(\Omega).
\end{equation}
We first observe that by choosing $\varphi$ to be the characteristic function of $\Omega$, we obtain
\[
\delta_1^2\,\int_\Omega u_1\,dx=\delta_0^2\,\int_\Omega u_0\,dx,
\]
and thus  we must have $u_0\not\equiv u_1$.
\par
We then choose the test function $\varphi=(u_1-u_0)_+$ in \eqref{subtract}, which implies
\[
\int_\Omega |\nabla (u_1-u_0)_+|^2\,dx+\int_\Omega (\delta_1^2\,u_1-\delta_0^2\,u_0)\,(u_1-u_0)_+\,dx=0.
\]
This can be rearranged into
\[
\int_\Omega |\nabla (u_1-u_0)_+|^2\,dx+\delta_1^2\,\int_\Omega (u_1-u_0)\,(u_1-u_0)_+\,dx=(\delta_0^2-\delta_1^2)\,\int_\Omega u_0\,(u_1-u_0)_+\,dx.
\]
The left-hand side is non-negative, thus by using that $\delta_1>\delta_0>0$, we get that
\[
\int_\Omega u_0\,(u_1-u_0)_+\,dx\le 0.
\]
Since $u_0>0$ by Remark \ref{rem:smoothness}, the previous estimate entails that $(u_1-u_0)_+$ must vanish almost everywhere.
This shows that
\[
u_1\le u_0,\qquad \mbox{ a.\,e. in }\Omega.
\]
By the interior smoothness of both functions, this property must actually hold everywhere.
\par
Finally, in order to obtain the strict sign, it is sufficient to observe that the difference $v=u_0-u_1$ is a non-negative smooth solution of
\[
-\Delta v+\delta_0^2\,v=(\delta_1^2-\delta_0^2)\,u_1\ge 0.
\]
By using again the minimum principle, we get that
\[
\mbox{ either }\quad v\equiv 0\quad \mbox{ or }\quad v>0.
\]
However, the first fact has been already excluded at the beginning. This concludes the proof.
\end{proof}
\begin{thm}[Asymptotics for $\delta\to 0$]
\label{thm:asymptotics}
Let $\Omega\subset\mathbb{R}^N$ be an open bounded connected set, with Lipschitz boundary. Then we have
\begin{equation}
\label{asymptotics}
\lim_{\delta\to 0^+}\|\nabla (\delta^2\,u_{\Omega,\delta})\|_{L^2(\Omega)}=0,
\end{equation}
and
\begin{equation}
\label{asymptotics2}
\lim_{\delta\to 0^+} \left\|\delta^2\,u_{\Omega,\delta}-\frac{\mathcal{H}^{N-1}(\partial\Omega)}{|\Omega|}\right\|_{L^m(\Omega)}=0,\qquad \mbox{ for every } 2\le m<\infty.
\end{equation}
Moreover, it also holds
\[
\lim_{\delta\to 0^+}\delta^2\, T(\Omega;\delta)=\frac{(\mathcal{H}^{N-1}(\partial\Omega))^2}{|\Omega|}.
\]
\end{thm}
\begin{proof}
From the identity
\begin{equation}
\label{hey!}
\int_\Omega |\nabla u_{\Omega,\delta}|^2\,dx+\delta^2\,\int_\Omega (u_{\Omega,\delta})^2\,dx=\int_{\partial\Omega} u_{\Omega,\delta}\,d\mathcal{H}^{N-1},
\end{equation}
we get in particular, for every $0<\delta\le 1$,
\begin{equation}
\label{gradient}
\begin{split}
\int_\Omega |\nabla (\delta^2\,u_{\Omega,\delta})|^2\,dx\le \delta^4\,\int_{\partial\Omega} u_{\Omega,\delta}\,d\mathcal{H}^{N-1}&=\delta^4\, T(\Omega;\delta)\\
&\le \frac{\delta^4}{\min\{1,\delta^2\}}\,\frac{1}{\eta_1(\Omega)}= \frac{\delta^2}{\eta_1(\Omega)}.
\end{split}
\end{equation}
Observe that we used formula \eqref{torsionintegral} and the upper bound \eqref{upperbase}. This shows \eqref{asymptotics}.
\par
Moreover, still from \eqref{hey!}, we also have
\[
\int_{\Omega} (\delta^2\,u_{\Omega,\delta})^2\,dx\le \delta^2\,\int_{\partial\Omega} u_{\Omega,\delta}\,d\mathcal{H}^{N-1}\le \frac{1}{\eta_1(\Omega)}.
\]
These show that the family of functions $\{\delta^2\,u_{\Omega,\delta}\}_{0<\delta\le 1}\subset W^{1,2}(\Omega)$ is equi-bounded. By the Rellich-Kondra\v{s}ov Theorem, for every infinitesimal descreasing sequence $\{\delta_n\}_{n\in\mathbb{N}}$ there exists a function $w\in W^{1,2}(\Omega)$ such that (up to a subsequence)
\[
\lim_{n\to\infty} \|\delta_n^2\,u_{\Omega,\delta_n}-w\|_{L^2(\Omega)}=0,
\]
and
\[
\lim_{n\to\infty} \int_\Omega \langle\nabla (\delta_n^2\,u_{\Omega,\delta_n}),\phi\rangle\,dx=\int_\Omega \langle \nabla w,\phi\rangle\,dx,\qquad \mbox{ for every } \phi\in L^2(\Omega;\mathbb{R}^N).
\]
Moreover, thanks to \eqref{gradient}, we actually get that it must result $\nabla w=0$ almost everywhere in $\Omega$. The connectedness assumption entails that $w$ is constant on $\Omega$. In order to identify the value of such a constant, we recall that from \eqref{L1} we have
\[
\int_\Omega (\delta_n^2\,u_{\Omega_n,\delta_n})\,dx=\mathcal{H}^{N-1}(\partial\Omega).
\]
By taking the limit as $n$ goes to $\infty$ in this identity, we finally get
\[
\int_\Omega w\,dx=\mathcal{H}^{N-1}(\partial\Omega).
\]
By recalling that $w$ is constant, we get
\[
w=\frac{\mathcal{H}^{N-1}(\partial\Omega)}{|\Omega|}.
\]
We now observe that the limit function $w$ is uniquely determined and thus does not depend on the particular sequence $\{\delta_n\}_{n\in\mathbb{N}}$. Thus, we can finally infer convergence in $L^2$ of the whole family
$\{\delta^2\,u_{\Omega,\delta}\}_{0<\delta\le 1}$, i.e. we showed \eqref{asymptotics2} for $m=2$.
\par
In order to get \eqref{asymptotics2} for an exponent $2<m<\infty$, it is sufficient to observe that
\[
\begin{split}
\left\|\delta^2\,u_{\Omega,\delta}-\frac{\mathcal{H}^{N-1}(\partial\Omega)}{|\Omega|}\right\|_{L^m(\Omega)}&\le \left\|\delta^2\,u_{\Omega,\delta}-\frac{\mathcal{H}^{N-1}(\partial\Omega)}{|\Omega|}\right\|_{L^\infty(\Omega)}^{1-\frac{2}{m}}\\
&\times \left\|\delta^2\,u_{\Omega,\delta}-\frac{\mathcal{H}^{N-1}(\partial\Omega)}{|\Omega|}\right\|_{L^2(\Omega)}^\frac{2}{m}.
\end{split}
\]
The last term converges to $0$ thanks to the proof above, while the $L^\infty$ norm is uniformly bounded, thanks to the fact that by Proposition \ref{prop:ubasic} we have for every $0<\delta\le 1$
\[
\begin{split}
\|\delta^2\,u_{\Omega,\delta}\|_{L^\infty(\Omega)}\le \|\delta^2\,u_{\Omega,\delta}\|_{L^\infty(\partial\Omega)}&\le C_q\,\left(\frac{(\delta^2\,T(\Omega;\delta))^\frac{q-2}{q}}{\eta_{q}(\Omega)}\right)^\frac{q}{2\,(q-1)}\\
&\le C_q\,\left(\frac{\eta_1(\Omega)^\frac{2-q}{q}}{\eta_{q}(\Omega)}\right)^\frac{q}{2\,(q-1)},
\end{split}
\]
for a fixed $2<q<2^\#$.
Observe that we also used the upper bound in \eqref{upperbase}.
\par
Finally, the last part of the statement easily follows by taking the limit in formula \eqref{torsionalproblem}.
\end{proof}
\begin{rem}
\label{rem:sharpbase}
With the previous result at hand, we can now prove that the lower bound
\[
\frac{1}{\delta^2}\,\frac{(\mathcal{H}^{N-1}(\partial\Omega))^2}{|\Omega|}\le T(\Omega;\delta),
\]
proved in \eqref{upperbase} is actually {\it sharp}, for every fixed $\delta>0$. Indeed, let us take $\Omega\subset\mathbb{R}^N$ an open bounded connected set, with Lipschitz boundary. By the scaling law \eqref{scaling}, we have
\[
T(t\,\Omega,\delta)=t^N\,T(\Omega;t\,\delta),\qquad \mbox{ for every } t>0.
\]
We thus obtain
\[
\begin{split}
\lim_{t\to 0^+}\delta^2\,T(t\,\Omega;\delta)\,\frac{|t\,\Omega|}{(\mathcal{H}^{N-1}(\partial(t\,\Omega)))^2}&=\lim_{t\to 0^+}\delta^2\,t^N\,T(\Omega;t\,\delta)\,\frac{|\Omega|\,t^N}{(\mathcal{H}^{N-1}(\partial \Omega)\,t^{N-1})^2}\\
&=\lim_{t\to 0^+}(\delta\,t)^2\,T(\Omega;t\,\delta)\,\frac{|\Omega|}{(\mathcal{H}^{N-1}(\partial \Omega))^2}=1,
\end{split}
\]
thanks to Theorem \ref{thm:asymptotics}.
\end{rem}

\section{Exact solutions in some special sets}
\label{sec:4}

In what follows, we will by indicate by $I_\alpha$ and $K_\alpha$ the {\it modified Bessel functions with index $\alpha$} of first and second kind, respectively.
We recall that these functions have the following asymptotic behavior for $z$ converging to $0$
\begin{equation}
\label{Bessel-behavior}
		I_{\alpha}(z)\sim\frac{1}{\Gamma(\alpha+1)}\left(\frac{z}{2}\right)^{\alpha},
\end{equation}
and
\begin{equation}
\label{Bessel-behavior2}
		K_{\alpha}(z)\sim
		\left\{\begin{array}{lr}
			-\log\left(\dfrac{z}{2}\right),& \text{ for }\alpha = 0,\\
			&\\
			\dfrac{\Gamma(\alpha)}{2}\,\left(\dfrac{2}{z}\right)^\alpha,& \text{ otherwise}.
		\end{array}\right.
\end{equation}
Here $\Gamma$ is the usual Gamma function. We refer to \cite{NIST} or \cite[Chapter 5]{Le} for the general properties of these functions.
\begin{lem}[Balls]
\label{lem:ball}
Let $\delta>0$ and let $B\subset\mathbb{R}^N$ be the ball of radius $1$, centered at the origin. Then $u_{B,\delta}$ is a radially symmetric increasing function.
This is explicitly given by
\begin{equation}
\label{uball}
	u_{B,\delta}(x)=\mathcal{U}_{\delta}(|x|)\qquad \mbox{ where }\quad\mathcal{U}_{\delta}(\varrho)=\frac{\varrho^{1-N/2}\,I_{N/2-1}(\delta\, \varrho)}{\delta\,I_{N/2}(\delta)}, \mbox{ for } 0\le \varrho<1.
\end{equation}
Accordingly, we get
\begin{equation}
	\label{torsiondisc}
	T(B;\delta) = \int_{\partial B}{u_{B,\delta}}\,d\mathcal{H}^{N-1}=\frac{N\,\omega_N\,I_{N/2-1}(\delta)}{\delta\,I_{N/2}(\delta)}.
\end{equation}
Finally, the function $\varrho\mapsto \mathcal{U}'_{\delta}(\varrho)/\varrho$ is monotone non-decreasing.
\end{lem}
\begin{proof}
Let $O$ be an $N\times N$ rotation matrix, we set $v(x)=u_{B,\delta}(O\,x)$. Thanks to the symmetries of $B$ and using the change of variable $y=O\,x$, we get
that $v$ is still a maximizer of \eqref{torsionalproblem}. By uniqueness, we thus obtain
\[
u_{B,\delta}(x)=u_{B,\delta}(O\,x),\qquad \mbox{ for } x\in B.
\]
By arbitrariness of the matrix $O$, this gives that $u_B$ must be radially symmetric. Thus there exists $\mathcal{U}_\delta$ such that
\[
u_{B,\delta}(x)=\mathcal{U}_\delta(|x|),
\]
and the function $\mathcal{U}_\delta$ must be the unique solution of
\begin{equation}
\label{1dmax}
\sup_{\varphi\in W^{1,2}(\mathrm{I})}\left\{2\,\varphi(1)-\int_\mathrm{I} |\varphi'(\varrho)|^2\,\varrho^{N-1}\,d\varrho-\delta^2\,\int_\mathrm{I} \varphi(\varrho)^2\,\varrho^{N-1}\,d\varrho\right\}
\end{equation}
where we set $\mathrm{I}=(0,1)$.
We now construct the new admissible function
\begin{equation}
\label{rearrange}
\widetilde{\mathcal{U}}_\delta(\varrho)=\left(\mathcal{U}_\delta(1)-\int_\varrho^1 (\mathcal{U}'_\delta(t))_+\,dt\right)_+.
\end{equation}
\begin{figure}
\includegraphics[scale=.2]{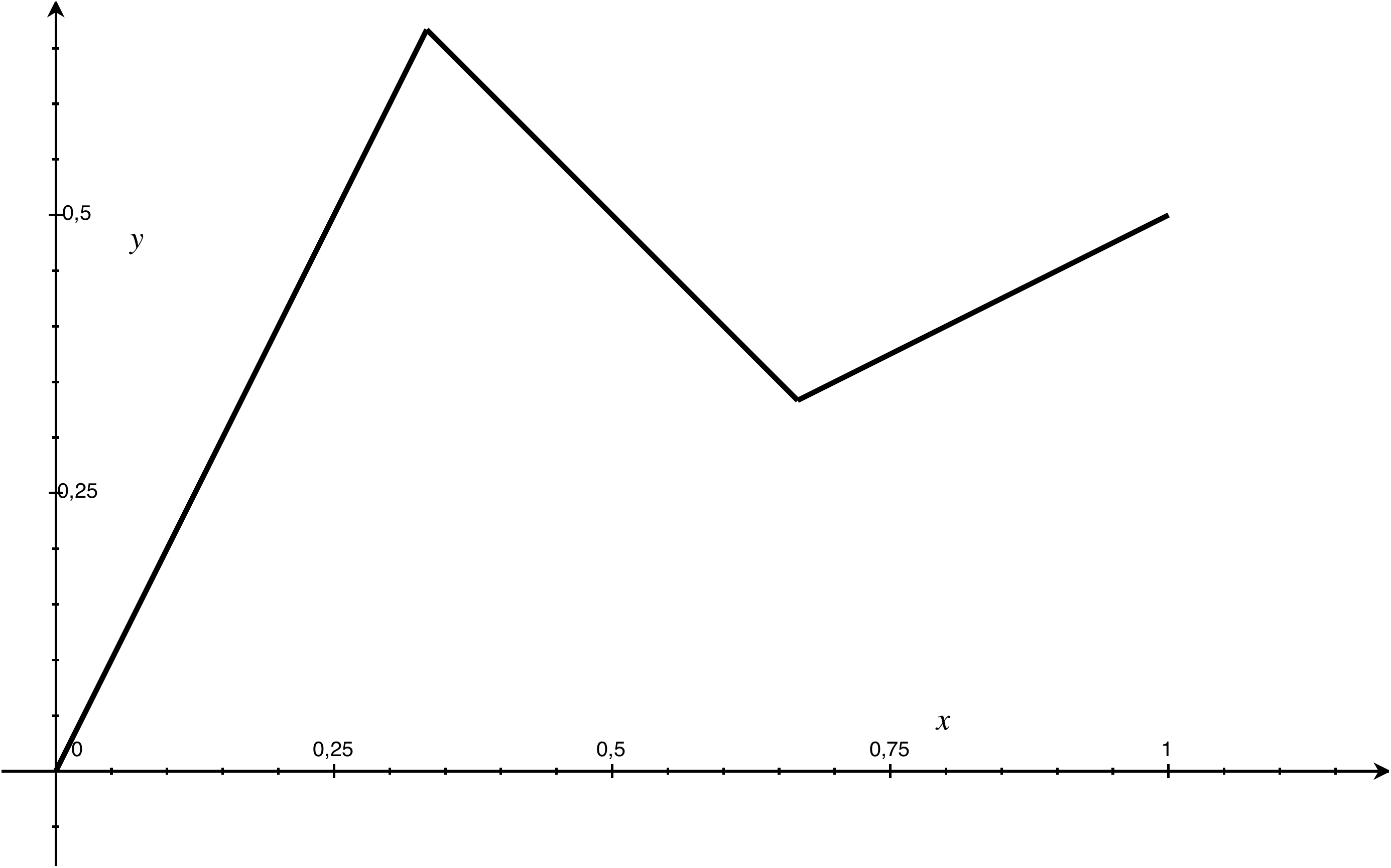}
\includegraphics[scale=.2]{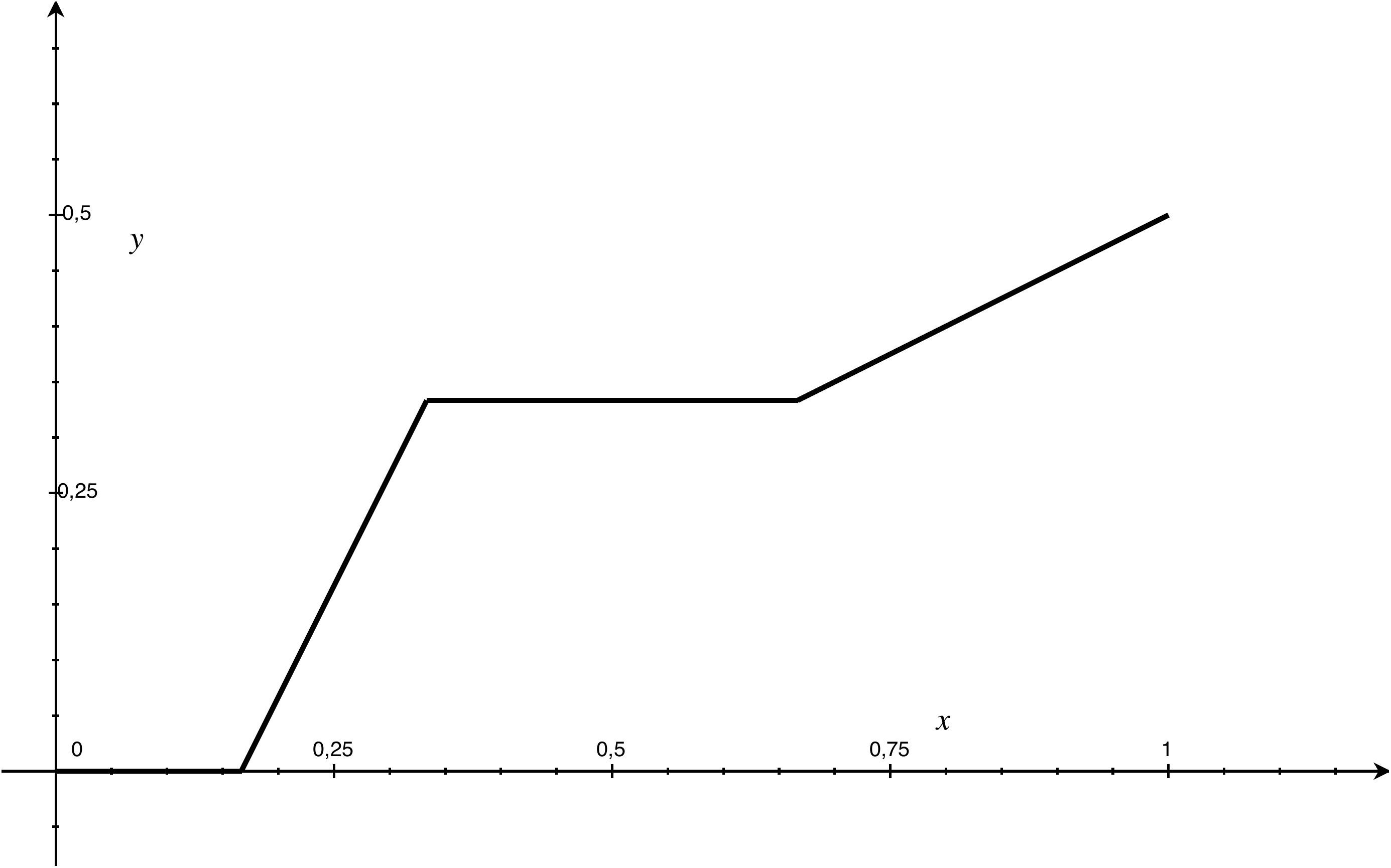}
\caption{An example of function (on the left) and its rearrangement (on the right) given by \eqref{rearrange}.}
\end{figure}
Observe that this is non-negative and non-decreasing by construction, since
\[
\frac{d}{d\varrho}\left(\mathcal{U}_\delta(1)-\int_\varrho^1 (\mathcal{U}'_\delta(t))_+\,dt\right)=(\mathcal{U}'_\delta(\varrho))_+\ge 0,
\]
and thus $\widetilde{\mathcal{U}}_\delta$ is the composition of two non-decreasing functions.
Moreover, we have $\widetilde{\mathcal{U}}_\delta(1)=\mathcal{U}_\delta(1)$ and
\[
\mathcal{U}_\delta(1)-\int_\varrho^1 (\mathcal{U}_\delta'(t))_+\,dt\le \mathcal{U}_\delta(1)-\int_\varrho^1\mathcal{U}_\delta'(t)\,dt=\mathcal{U}_\delta(\varrho).
\]
By taking the positive part on both sides and using that $\mathcal{U}_\delta\ge 0$, we obtain
\[
0\le \widetilde{\mathcal{U}}_\delta\le \mathcal{U}_\delta.
\]
Finally, by construction we also have
\[
|\widetilde{\mathcal{U}}'_\delta(\varrho)|\le |\mathcal{U}_\delta'(\varrho)|,\qquad \mbox{ for a.\,e. }\varrho\in \mathrm{I}.
\]
All these properties show that $\widetilde{\mathcal{U}}_\delta$ must be another maximizer of \eqref{1dmax}. By uniqueness, we  obtain $\widetilde{\mathcal{U}}_\delta=\mathcal{U}_\delta$ and thus, $\mathcal{U}_\delta$ has the claimed monotonicity.
\vskip.2cm\noindent
In order to explicitly determine $\mathcal{U}_\delta$, it is sufficient to write the optimality condition. This is given by the boundary value problem
\[
\left\{\begin{array}{rcll}
-\varrho^2\,\psi''-\varrho\,(N-1)\,\psi'+\delta^2\,\varrho^2\,\psi &=& 0,& \mbox{ in } \mathrm{I},\\
\psi'(0)&=&0,&\\
\psi'(1)&=&1.&
\end{array}
\right.
\]
The general solution of the ODE is given by (see for example \cite[Chapter 5, Section 7]{Le})
\begin{equation}
	\label{generalsolution}
\psi(\varrho) = C\,\varrho^{1-\frac{N}{2}}\,I_{N/2-1}(\delta\, \varrho)+D\,\varrho^{1-\frac{N}{2}}\,K_{N/2-1}(\delta\, \varrho),\qquad C,D\in\mathbb{R}.
\end{equation}
Thus, as $K_\alpha$ is singular at the origin by \eqref{Bessel-behavior2}, the condition $w'(0)=0$ imposes that $D=0$. This leads to
\[
\mathcal{U}_\delta(\varrho)= C\,\varrho^{1-\frac{N}{2}}\,I_{N/2-1}(\delta\, \varrho).
\]
We impose the other boundary condition $w'(1)=1$ in order to fix the constant $C$. By recalling the relation (see \cite[equation (5.7.9), page 110]{Le})
\begin{equation}
\label{magic}
\frac{d}{d\varrho} (\varrho^{-\alpha}\,I_{\alpha}(\varrho))=\varrho^{-\alpha}\,I_{\alpha+1}(\varrho),
\end{equation}
and using this with $\alpha=N/2-1$, we finally get the expression \eqref{uball}.
\par
We can now substitute this expression in \eqref{torsionintegral} in order to find the explicit formula \eqref{torsiondisc} for the torsional rigidity.
\vskip.2cm\noindent
At last, we prove that
\[
\mathcal{W}(\varrho):=\frac{\mathcal{U}'_{\delta}(\varrho)}{\varrho}=\frac{1}{\varrho}\,\frac{d}{d\varrho}\frac{\varrho^{1-N/2}\,I_{N/2-1}(\delta\, \varrho)}{\delta\,I_{N/2}(\delta)},
\]
is monotone non-decreasing. By using \eqref{magic}, it is easy to see that this can be re-written as
\[
\mathcal{W}(\varrho)=\frac{\varrho^{-N/2}\,I_{N/2}(\delta\, \varrho)}{I_{N/2}(\delta)}.
\]
Thus $\mathcal{W}$ solves the ODE
\[
-\varrho^2\,\mathcal{W}''-\varrho\,(N+1)\,\mathcal{W}'+\delta^2\,\varrho^2\,\mathcal{W} = 0,\qquad \mbox{ in } \mathrm{I},
\]
with boundary conditions
\[
\mathcal{W}'(0)=0\qquad \mbox{ and }\qquad \mathcal{W}'(1)=\delta\,\frac{I_{N/2+1}(\delta)}{I_{N/2}(\delta)}=:c_\delta.
\]
We notice that these can be inferred by appealing again to \eqref{magic} and \eqref{Bessel-behavior}. These properties and the first part of the proof imply that, if we denote by $B'\subset \mathbb{R}^{N+2}$ the $(N+2)-$dimensional unit ball centered at the origin, the radially symmetric function
\[
w(x):=\mathcal{W}(|x|),\qquad \mbox{ for }x\in B',
\]
is the unique solution of
\[
\left\{\begin{array}{rcll}
-\Delta u+\delta^2\,u&=&0,&\text{in }B',\\
\langle \nabla u,\nu_\Omega\rangle &=&c_\delta,&\text{on }\partial B'.
\end{array}\right.
\]
Thus, it is the unique maximizer of
\[
\sup_{\varphi\in W^{1,2}(B')}\left\{2\,c_\delta\,\int_{\partial B'} \varphi\,d\mathcal{H}^{N+1} -\int_{B'}|\nabla \varphi|^2\,dx-\delta^2\,\int_{B'}\varphi^2\,dx\right\}.
\]
By recalling that $I_\alpha$ is always positive for real arguments, we get that $c_\delta>0$ and thus the claimed monotonicity of $w$ now follows by using the same argument as for \eqref{1dmax}.
\end{proof}
\begin{rem}
By recalling the scaling laws \eqref{scaling} and \eqref{scaling_u}, for a generic ball of radius $R>0$ we have
\[
u_{B_R,\delta}(x)=R\,u_{B,\delta\,R}\left(\frac{x}{R}\right)=\frac{|x|^{1-\frac{N}{2}}\,I_{N/2-1}(\delta\, |x|)}{\delta\,R^{1-\frac{N}{2}}\,I_{N/2}(\delta\, R)},
\]
and
\[
T(B_R;\delta)=R^N\,T\left(B;\delta\,R\right)=R^{N-1}\,\frac{N\,\omega_N\,I_{N/2-1}(\delta\,R)}{\delta\,I_{N/2}(\delta\,R)}.
\]
\end{rem}
\begin{lem}[Spherical shells]
Let $\delta>0$. For $0<r<R$, we consider the spherical shell
\[
\Omega=\Big\{x\in\mathbb{R}^N\, :\, r<|x|<R\Big\}.
\]
Then $u_{\Omega,\delta}$ is a radially symmetric function.
This is explicitly given by
\[
	u_{\Omega,\delta}(x)=\mathcal{V}_{r,R,\delta}(|x|),\qquad \mbox{ for } x\in \Omega,
\]
where
\[
\mathcal{V}_{r,R,\delta}(\varrho)=C_0\,\varrho^{1-\frac{N}{2}}\,I_{N/2-1}(\delta\, \varrho)+D_0\,\varrho^{1-\frac{N}{2}}\,K_{N/2-1}(\delta\, \varrho),
\]
and the constants $C_0=C_0(r,R,\delta)\not=0$ and $D_0=D_0(r,R,\delta)\not =0$ are given by
\[
	\begin{array}{ccc}
		C_0 &=& \dfrac{r^{1-\frac{N}{2}}\,K_{N/2}(\delta\, r)+R^{1-\frac{N}{2}}\,K_{N/2}(\delta\, R)}{\delta\, r^{1-\frac{N}{2}}\,R^{1-\frac{N}{2}}\Big(I_{N/2}(R)\,K_{N/2}(r)-I_{N/2}(r)\,K_{N/2}(R)\Big)},\\
		&&\\
		D_0 &=& \dfrac{r^{1-\frac{N}{2}}\,I_{N/2}(\delta\, r)+R^{1-\frac{N}{2}}\,I_{N/2}(\delta\, R)}{\delta\, r^{1-\frac{N}{2}}\,R^{1-\frac{N}{2}}\Big(I_{N/2}(r)\,K_{N/2}(R)-I_{N/2}(R)\,K_{N/2}(r)\Big)}.
	\end{array}
\]
Accordingly, we get
\[
\begin{split}
		T(\Omega;\delta)&=\frac{\left[r^{1-\frac{N}{2}}\,K_{N/2}(\delta\, r)+R^{1-\frac{N}{2}}\,K_{N/2}(\delta\, R)\right]\,\left[R^{1-\frac{N}{2}}I_{1-N/2}(\delta\, R)+r^{1-\frac{N}{2}}\,I_{1-N/2}(\delta\, r)\right]}{\delta\, r^{1-\frac{N}{2}}\,R^{1-\frac{N}{2}}\,\left[I_{N/2}(R)\,K_{N/2}(r)-I_{N/2}(r)\,K_{N/2}(R)\right]}\\
		&+\frac{\left[r^{1-\frac{N}{2}}\,I_{N/2}(\delta\, r)+R^{1-\frac{N}{2}}\,I_{N/2}(\delta\, R)\right]\,\left[R^{1-\frac{N}{2}}\,K_{1-N/2}(\delta \,R)+r^{1-\frac{N}{2}}\,K_{1-N/2}(\delta\, r)\right]}{\delta\, r^{1-\frac{N}{2}}\,R^{1-\frac{N}{2}}\,\left[I_{N/2}(r)\,K_{N/2}(R)-I_{N/2}(R)\,K_{N/2}(r)\right]}.
\end{split}
\]
\end{lem}
\begin{proof}
The radial symmetry of $u_{\Omega,\delta}$ can be obtained with the same proof of Lemma \ref{lem:ball}. In order to determine $u_{\Omega,\delta}$, we proceed as before, by seeking this time the solution of
\[
\left\{\begin{array}{rccl}
-\varrho^2\,w''-\varrho\,(N-1)\,w'+\delta^2\,\varrho^2\,w &=& 0,& \mbox{ in } (r,R),\\
w'(r)&=&-1,&\\
w'(R)&=&1.&
\end{array}
\right.
\]
By recalling the expression of the general solution \eqref{generalsolution}, we can calculate the explicit form of both constants $C$ and $D$ by imposing the two boundary conditions. Once $u_{\Omega,\delta}$ is obtained, it is sufficient to use again \eqref{torsionintegral} to get $T(\Omega;\delta)$, as well. We leave the details to the reader.
\end{proof}

\begin{lem}[Hyperrectangle]
\label{lem:hyper}
Let $\delta>0$ and let $\ell_1,\ell_2,\dots,\ell_N>0$. If we set
\[
\Omega=\prod_{i=1}^N(-\ell_i,\ell_i),
\]
then we have
\begin{equation}
\label{hyper}
u_{\Omega,\delta}(x)=\sum_{i=1}^N \frac{\cosh(\delta\,x_i)}{\delta\,\sinh(\delta\,\ell_i)},\qquad \mbox{ for every } x=(x_1,\dots,x_N)\in\Omega.
\end{equation}
Its boundary torsional rigidity is given by
\[
T(\Omega;\delta)
=\sum_{k=1}^N \left[\frac{1}{\delta\,\tanh(\delta\,\ell_k)}\,\mathcal{H}^{N-1}(\Sigma_k)+\sum_{i\not=k} \frac{1}{\delta^2}\,\mathcal{H}^{N-2}(\Sigma_{k,i})\right],
\]
where
\[
\Sigma_k=\Big\{x\in\overline\Omega\, :\, |x_k|=\ell_k\Big\}\qquad \mbox{ and }\qquad \Sigma_{k,i}=\Big\{x\in\overline\Omega\, :\, |x_k|=\ell_k,\, |x_i|=\ell_i\Big\}.
\]
\end{lem}
\begin{proof}
One can directly verify that \eqref{hyper} is a classical solution of the optimality condition \eqref{torsionalproblem}. In order to compute the boundary torsional rigidity, we use \eqref{hyper} in formula \eqref{torsionintegral}. We get
\[
\begin{split}
T(\Omega;\delta)&=\int_{\partial\Omega} u_{\Omega,\delta}\,d\mathcal{H}^{N-1}=\sum_{k=1}^N \int_{\Sigma_k} \left[\frac{\cosh(\delta\,x_k)}{\delta\,\sinh(\delta\,\ell_k)}+\sum_{i\not=k}\frac{\cosh(\delta\,x_i)}{\delta\,\sinh(\delta\,\ell_i)}\right]\,d\mathcal{H}^{N-1}.
\end{split}
\]
We then observe that for every $k\in\{1,\dots,N\}$, we have (observe that $\cosh(\delta\,x_k)$ is constant on $\Sigma_k$)
\[
\int_{\Sigma_k} \frac{\cosh(\delta\,x_k)}{\delta\,\sinh(\delta\,\ell_k)}\,d\mathcal{H}^{N-1}=\frac{\cosh(\delta\,\ell_k)}{\delta\,\sinh(\delta\,\ell_k)}\,\mathcal{H}^{N-1}(\Sigma_k),
\]
and for $i\not=k$ we get
\[
\int_{\Sigma_k} \frac{\cosh(\delta\,x_i)}{\delta\,\sinh(\delta\,\ell_i)}\,d\mathcal{H}^{N-1}=\frac{4}{\delta^2}\,\prod_{j\not=i,k} (2\,\ell_j)=\frac{\mathcal{H}^{N-2}(\Sigma_{k,i})}{\delta^2}.
\]
This concludes the proof.
\end{proof}

\section{Geometric estimates}
\label{sec:5}

\subsection{Planar simply connected sets}

We start by proving a lower bound for the boundary torsional rigidity by using {\it conformal transplantation}, when $N=2$. As we have mentioned in the introduction, this technique has been successfully employed in order to give geometric estimates for eigenvalues of planar sets, see for example \cite{Po55, PolyaandSzego, Sze,Wei}.
\par
Roughly speaking, it consists in producing trial functions for a variational problem in an open simply connected set $\Omega\subset\mathbb{R}^2$, by simply composing functions defined on the unit disk $\mathbb{D}=\{x\in\mathbb{R}^2\, :\,|x|<1\}$ with an holomorphic map $h:\Omega\to \mathbb{D}$.
This is particularly useful when tackling problems involving the Dirichlet integral, due to its conformal invariance.\\

In order to state our lower bound, let us recall a few facts on conformal mappings.
The celebrated {\it Riemann Mapping Theorem} (see \cite[Chapter 6]{ahlfors}) states that given any simply connected region $\Omega\subsetneq \mathbb R^2$ and a point $x_0\in\Omega$, there exists a unique (up to a rotation) holomorphic isomorphism
\[
f_{x_0}:\mathbb{D}\to\Omega,\qquad \mbox{ with }\ f_{x_0}(0)=x_0.
\]
Furthermore, when $\partial\Omega$ is $C^{1,\alpha}$, we know that this mapping can be extended up to the boundary, it is $C^1$ in $\overline{\mathbb{D}}$ and we have
\[
f'_{x_0}(x)\not =0,\qquad \mbox{ for every } x\in\partial\mathbb{D}.
\]
We refer to \cite[Theorem 1 \& Theorem 2]{Wa} for this result.
\begin{defn}\label{def:boundary-distortion}
Let $\Omega\subsetneq\mathbb{R}^2$ be an open bounded simply connected set, with $C^{1,\alpha}$ boundary, for some $0<\alpha\le 1$. With the previous notation,
we define the {\it boundary distortion radius of $\Omega$} by
\[
\dot{\mathcal{R}}_\Omega:=\inf_{x_0\in\Omega}\left(\frac{1}{2\,\pi}\,\int_{\partial \mathbb{D}} |f'_{x_0}|^2\,d\mathcal{H}^1\right)^\frac{1}{2}.
\]
It is readily seen from its definition that this quantity scales like a length.
\end{defn}
\begin{rem}
\label{rem:mono}
By using that $|f'_{x_0}|^2$ is a subharmonic function, we get that the map
\[
\varrho\mapsto \frac{1}{2\,\pi\,\varrho}\,\int_{\{|x|=\varrho\}} |f'_{x_0}|^2\,d\mathcal{H}^1,
\]
is monotone non-decreasing, a result originally due to G. H. Hardy, see \cite[Theorem III]{Ha}. In particular, we have
\[
|f'_{x_0}(0)|\le \left(\frac{1}{2\,\pi\,\varrho}\,\int_{\{|x|=\varrho\}} |f'_{x_0}|^2\,d\mathcal{H}^1\right)^\frac{1}{2}\le \left(\frac{1}{2\,\pi}\,\int_{\partial \mathbb{D}} |f'_{x_0}|^2\,d\mathcal{H}^1\right)^\frac{1}{2}.
\]
Then the definition of boundary distortion radius is maybe better appreciated by recalling that
\[
\dot{r}_\Omega:=\sup_{x_0\in\Omega} |f'_{x_0}(0)|,
\]
is usually called {\it conformal radius of $\Omega$}. This quantity naturally appears in many geometric estimates for the spectrum of the Laplacian of a planar simply connected set, see for example \cite{Po55} and \cite{PolyaandSzego}.
\par
Finally, we may notice that the quantity
\[
\left(\frac{1}{2\,\pi}\,\int_{\partial \mathbb{D}} |f'_{x_0}|^2\,d\mathcal{H}^1\right)^\frac{1}{2},
\]
can be
recognized as the norm of the conformal map $f'_{x_0}$ in the Hardy space $\mathrm{H}^2(\mathbb{D})$. We refer the reader to \cite{Koo} for an introduction to these spaces.
\end{rem}
\begin{lem}
\label{lem:simple}
Let $\Omega\subset\mathbb{R}^2$ be a bounded simply connected open set, with $C^{1,\alpha}$ boundary, for some $0<\alpha\le 1$. Then we have
\begin{equation}
\label{area}
|\Omega|\le \pi\,\dot{\mathcal{R}}_\Omega^2.
\end{equation}
In particular, if $\Omega$ is a disk of radius $R$, we have $\dot{\mathcal{R}}_\Omega=R$ and the equality holds in \eqref{area}.
\end{lem}
\begin{proof}
Let $x_0\in \Omega$, with the notation above we recall that $|f'_{x_0}|^2$ coincides with the Jacobian determinant of $f_{x_0}$, seen as a two-dimensional change of variables. Thus we have
\[
|\Omega|=\int_{\mathbb{D}} |f'_{x_0}|^2\,dw.
\]
We can then write
\[
\int_{\mathbb{D}} |f'_{x_0}|^2\,dw=\int_0^1 \left(\int_{\{|x|=\varrho\}} |f_{x_0}'|^2\,d\mathcal{H}^1\right)\,d\varrho.
\]
By using Remark \ref{rem:mono} and the arbitrariness of $x_0\in\Omega$, we then get the desired estimate \eqref{area}.
\par
As for the second statement, we observe that for a disk $B_R(\xi_0)$, by choosing $x_0=\xi_0$ and
\[
f_{\xi_0}(x)=R\,x+\xi_0,
\]
we have $|f'_{\xi_0}|=R$ and thus
\[
\dot{\mathcal{R}}_{B_R(\xi_0)}\le \left(\frac{1}{2\,\pi}\,\int_{\partial \mathbb{D}} |f'_{\xi_0}|^2\,d\mathcal{H}^1\right)^\frac{1}{2}=R.
\]
The reverse inequality follows from \eqref{area}. This concludes the proof.
\end{proof}
We are now ready for the main result of this subsection.
\begin{thm}
\label{thm:lower-bound-plane}
Let $\delta>0$ and let $\Omega\subset\mathbb{R}^2$ be a bounded simply connected open set, with $C^{1,\alpha}$ boundary, for some $0<\alpha\le 1$. Then it holds
	\begin{equation}
		\label{lower-bound-plane}
\left(\frac{\mathcal{H}^1(\partial\Omega)}{2\,\pi}\right)^2\,\frac{T\left(B_{\dot{\mathcal{R}}_\Omega};\delta\right)}{\dot{\mathcal{R}}_\Omega^2}\le	T(\Omega;\delta).
	\end{equation}
	Moreover, equality holds if and only $\Omega$ is a disk.
\end{thm}
\begin{proof}
We first give the estimate under the assumption that $\dot{\mathcal{R}}_\Omega=1$. By definition, for every $\varepsilon>0$, there exists $x_0\in\Omega$ such that
\begin{equation}
\label{hope}
\frac{1}{2\,\pi}\,\int_{\partial \mathbb{D}} |f'_{x_0}|^2\,d\mathcal{H}^1< 1+\varepsilon.
\end{equation}
For simplicity, we indicate by $u$ the boundary $\delta-$torsion function of the unit disk $\mathbb{D}$, given by \eqref{uball}. With the notation above, we use the test function $\widetilde{u}=u\circ h_{x_0}$, where we have set
\[
h_{x_0}:=f^{-1}_{x_0}:\overline{\Omega}\to\overline{\mathbb{D}}.
\]
Observe that $\widetilde{u}\in W^{1,2}(\Omega)$, thanks to the properties of $f_{x_0}$.
This yields
\begin{equation}
\label{reciprocal}
		\frac{1}{T(\Omega;\delta)}\leq \frac{\displaystyle\int_{\Omega}|\nabla \widetilde u|^2\,dx+ \delta^2\,\int_{\Omega}\widetilde u^2\,dx}{\displaystyle\left(\int_{\partial\Omega}{\widetilde u\,d\mathcal{H}^1}\right)^2}=\left(\frac{\delta\, I_1(\delta)}{I_0(\delta)}\right)^2\,\frac{\displaystyle\int_{\mathbb{D}}{|\nabla u|^2\,dw}+ \delta^2\,\int_{\mathbb{D}}u^2\,|f'_{x_0}(w)|^2\,dw}{\displaystyle (\mathcal{H}^1(\partial\Omega))^2},
	\end{equation}
	where we used that $h_{x_0}:\partial\Omega\to \partial \mathbb{D}$ and that $u$ is radially symmetric so that, by \eqref{uball},
\[
\widetilde u=\frac{I_0(\delta)}{\delta\,I_1(\delta)}, \qquad \mbox{ on } \partial \Omega.
\]
In order to estimate the integral
\[
\int_{\mathbb{D}}u^2\,|f_{x_0}'(w)|^2\,dw=\int_0^1 u^2\,\left(\int_{\{|w|=\varrho\}} |f_{x_0}'|^2\,d\mathcal{H}^1\right)\,d\varrho,
\]
we set
\[
\Phi(\varrho)=\frac{1}{2\,\pi\,\varrho}\,\int_{\{|w|=\varrho\}} |f'_{x_0}|^2\,d\mathcal{H}^1,
\]
then we can rewrite
\[
\int_{\mathbb D}u^2\,|f_{x_0}'(w)|^2\,dw=2\,\pi\,\int_0^1 u^2\,\Phi(\varrho)\,\varrho\,d\varrho.
\]
From Remark \ref{rem:mono} we know that $\varrho\mapsto \Phi(\varrho)$ is monotone non-decreasing. Thus we obtain
\[
\int_{\mathbb D}u^2\,|f'_{x_0}(w)|^2\,dw\le \left(2\,\pi\,\int_0^1 u^2\,\varrho\,d\varrho\right)\,\Phi(1)=\Phi(1)\,\int_\mathbb{D} u^2\,dx.
\]
By recalling \eqref{hope}, this finally gives
\[
\int_{\mathbb D}u^2\,|f'_{x_0}(w)|^2\,dw<(1+\varepsilon)\,\int_\mathbb{D} u^2\,dx.
\]
We insert this estimate into \eqref{reciprocal} and use that $u$ is optimal for the disk. We get
\[
\frac{1}{T(\Omega;\delta)}\le \left(\frac{\delta\, I_1(\delta)}{I_0(\delta)}\right)^2\,\frac{\displaystyle\,\int_{\partial \mathbb{D}} u\,d\mathcal{H}^1+\delta^2\,\varepsilon\,\int_{\mathbb{D}}u^2\,dw}{\displaystyle (\mathcal{H}^1(\partial\Omega))^2}.
\]
If we now let $\varepsilon$ goes to $0$ and use \eqref{torsionintegral}, we obtain
\[
\frac{1}{T(\Omega;\delta)}\le \left(\frac{\delta\, I_1(\delta)}{I_0(\delta)}\right)^2\,\frac{T(\mathbb{D};\delta)}{\displaystyle (\mathcal{H}^1(\partial\Omega))^2}.
\]
Finally, from \eqref{torsiondisc} we know that
\[
\frac{\delta\,I_1(\delta)}{I_0(\delta)}=\frac{2\,\pi}{T(\mathbb{D};\delta)},
\]
from which \eqref{lower-bound-plane} follows, for the case $\dot{\mathcal{R}}_\Omega=1$.
\par
 The general case can be obtained by scaling. Indeed, taking $t=1/\dot{\mathcal{R}_\Omega}$, the scaled set $\widetilde \Omega=t\,\Omega$ has unit boundary distortion radius. Thus, by using the first part of the proof with $\delta/t$ and \eqref{scaling}, we have
\[
T(\Omega;\delta)=t^{-2}\,T\left(\widetilde \Omega;\frac{\delta}{t}\right)\ge t^{-2}\, \left(\frac{\mathcal{H}^1(\partial\widetilde\Omega)}{2\,\pi}\right)^2\,T\left(\mathbb{D};\frac{\delta}{t}\right)=\left(\frac{\mathcal{H}^1(\partial\Omega)}{2\,\pi}\right)^2\,\frac{T\left(B_{\dot{\mathcal{R}}_\Omega};\delta\right)}{\dot{\mathcal{R}}_\Omega^2}.
\]
This concludes the proof of the inequality.
\vskip.2cm\noindent
We now come to the equality cases. We first observe that disks achieve equality in \eqref{lower-bound-plane}, thanks to Lemma \ref{lem:simple}.  Let us now suppose that $\Omega$ is such that equality holds in \eqref{lower-bound-plane}. For simplicity, we can assume again that $\dot{\mathcal{R}}_\Omega=1$. This means that equality must hold everywhere in the proof above. In particular, there exists $x_0\in\Omega$ such that $\widetilde{u}=u\circ h_{x_0}$ must be the boundary $\delta-$torsion function of $\Omega$. This in particular implies that $\widetilde{u}$ must solve
\[
-\Delta \widetilde{u}+\delta^2\,\widetilde{u}=0,\qquad \mbox{ in }\Omega.
\]
By computing the Laplacian
 of $\widetilde{u}$, which is given by
 \[
\Delta \widetilde u(x)=|h'_{x_0}(x)|^2\,\Delta u(h_{x_0}(x)),\qquad \mbox{ for } x\in\Omega,
\] and using that $u$ solves
\[
-\Delta u+\delta^2\,u=0,\qquad \mbox{ in }\mathbb{D},
\]
we then obtain that
\[
(1-|h'_{x_0}|^2)\,\delta^2\,\widetilde{u}=0,\qquad \mbox{ in }\Omega.
\]
Recalling that $u$ is positive, the last identity entails that we must have $|h'_{x_0}|\equiv 1$ in $\Omega$ and thus $|f'_{x_0}|\equiv 1$ in $\mathbb{D}$. This implies that $f'_{x_0}$ is actually an isometry and thus $\Omega$ is a disk.
\end{proof}

\subsection{Convex sets: lower bound}\label{section:lowebound}

The following technical lemma will be used in a while. We will use again the notation $\mathrm{I}=(0,1)$.
\begin{lem}
\label{lm:alpha_parameter}
	For every $\delta>0$, we have
		\begin{equation}\label{alpha functional}
		\alpha(\delta):=\sup_{\varphi \in W^{1,2}(\mathrm{I})}\frac{\left(\varphi(0)\right)^2}{\displaystyle\int_\mathrm{I} |\varphi'|^2\,dt+\delta^2\,\int_\mathrm{I} \varphi^2\;dt}=\frac{1}{\delta\,\tanh(\delta)}.
	\end{equation}
Moreover, the function
\[
u_\mathrm{I}(t)=\frac{1}{\delta}\,\left(\frac{\cosh(\delta\,t)}{\tanh(\delta)} -\sinh(\delta\, t)\right),
\]
attains the supremum in \eqref{alpha functional}.
\end{lem}

\begin{proof}
An argument as in \eqref{tminimization} yields that the maximization problem \eqref{alpha functional} can be rephrased as
	\begin{equation}\label{alpha variational}
		\alpha(\delta)=\sup_{\varphi\in W^{1, 2}(\mathrm{I})}\left\{2\,\varphi(0)-\int_\mathrm{I} |\varphi'|^2\,dt-\delta^2\,\int_\mathrm{I} \varphi^2\,dt\right\}.
	\end{equation}
	The existence of a (unique) maximizer $u_I$ for \eqref{alpha variational} is again easily proven by the Direct Method. Moreover, by concavity of the maximization problem, we have that $\varphi$ solves \eqref{alpha variational} if and only if it verifies the optimality condition
\[
		\psi(0)=\int_\mathrm{I} \varphi'\,\psi'\,dt +\delta^2\,\int_\mathrm{I} \varphi\,\psi\, dt, \qquad \mbox{ for every } \psi\in W^{1, 2}(\mathrm{I}).
\]
This is the weak formulation of
\[
		\left\{\begin{array}{rcll}
		-\varphi''+\delta^2\,\varphi&=& 0, &\mbox{in }\mathrm{I},\\
			\varphi'(0)&=&-1,&\\
			\varphi'(1)&=&0.&
		\end{array}\right.
\]
The latter is uniquely solved by the function
\[
t\mapsto \frac{1}{\delta}\,\left(\frac{\cosh(\delta\,t)}{\tanh(\delta)} -\sinh(\delta\, t)\right),
\] 	
which then coincides with the unique maximizer $u_I$. From \eqref{alpha variational} and the equation in weak form we get
\[
\alpha(\delta)=2\,u_I(0)-\int_\mathrm{I} |u_I'|^2\,dt-\delta^2\,\int_\mathrm{I}  u_I^2\,dt=u_\mathrm{I}(0).
\]
By using the expression $u_\mathrm{I}$, we conclude.
\end{proof}
We recall the definition of {\it inradius} of $\Omega$, i.e.
\[
r_\Omega=\sup_{x\in\Omega}d_\Omega(x).
\]
Observe that this is the radius of a largest ball inscribed in $\Omega$.
We can then prove the following sharp lower bound on $T(\Omega;\delta)$ for a convex set. This is somehow similar to an estimate by P\'olya, which provides a lower bound on the usual torsional rigidity in terms of volume and perimeter, see \cite{polya1960two}. We employ the same technique, i.e. the {\it method of interior parallels}, introduced by Makai \cite{makai1959bounds,Mak} and P\`olya \cite{polya1960two}.
\begin{thm}\label{thm-lower-bound}
Let $\delta>0$ and let $\Omega\subset\mathbb{R}^N$ be an open bounded convex set. We have the following estimate
\[
		T(\Omega;\delta)>\frac{\mathcal{H}^{N-1}(\partial\Omega)}{\delta\,\tanh{(\delta\, r_\Omega)}}.
\]
Moreover, the estimate is sharp in the following sense: we have
\[
\lim_{n\to\infty} \frac{T(\Omega_n;\delta)\,\tanh(\delta\,r_{\Omega_n})}{\mathcal{H}^{N-1}(\partial\Omega_n)}=\frac{1}{\delta},\qquad \mbox{ for  }\Omega_n:=(-n, n)^{N-1}\times(-1, 1).
\]
\end{thm}
\begin{proof}
We first prove the estimate under the assumption that $r_\Omega=1$.
We define a non-negative test function for the torsion functional \eqref{torsionfunctional} of the form
\[
		\varphi(x)=u_I(d_\Omega(x)),\qquad \mbox{ for } x\in\Omega,
\]
	where $u_I$ is the same function as in Lemma \ref{lm:alpha_parameter}. We denote
	\[
	\Omega_t=\Big\{x\in\Omega\, :\, d_\Omega(x)=t\Big\},\qquad \mbox{ for } t\in[0,1].
	\]
    Then, by using the Coarea Formula and the fact that $|\nabla d_\Omega|=1$ almost everywhere in $\Omega$, we get
    	\[
    \begin{split}
   	T(\Omega;\delta)\geq\frac{\left(\displaystyle\int_{\partial\Omega}\varphi\,d\mathcal{H}^{N-1}\right)^2}
    {\displaystyle\int_{\Omega} |\nabla \varphi|^2\,dx+ \delta^2\,\int_\Omega \varphi^2\,dx}&=\frac{(u_I(0))^2\,(\mathcal{H}^{N-1}(\partial\Omega))^2}{\displaystyle\int_\Omega |u'_I(d_\Omega(x))|^2\,dx+\delta^2\,\int_\Omega (u_I(d_\Omega(x))^2\,dx}\\
    &=\frac{(u_I(0))^2\,(\mathcal{H}^{N-1}(\partial\Omega))^2}{\displaystyle\int_0^1 \Big[|u'_I(t)|^2\,+\delta^2\,(u_I(t))^2\Big]\,\mathcal{H}^{N-1}(\partial \Omega_t)\,dt}\\
   &\ge \frac{(u_I(0))^2}{\displaystyle\int_0^1 \Big[|u'_I(t)|^2\,+\delta^2\,(u_I(t))^2\Big]\,dt}\,\mathcal{H}^{N-1}(\partial\Omega).
      \end{split}
    \]
 In the last inequality we used that
\begin{equation}
\label{lema surface difference}
\mathcal{H}^{N-1}(\partial\Omega_t)\le \mathcal{H}^{N-1}(\partial\Omega),\qquad \mbox{ for } t\in(0,r_\Omega),
\end{equation}
which follows from convexity, see for example \cite[Lemma 2.2.2]{BB}.
 Observe that this inequality is strict, for an open bounded convex set. If we now apply Lemma \ref{lm:alpha_parameter}, we get the desired estimate in the case $r_\Omega=1$.
  \par
 The general case now follows by scaling: indeed, by taking $t=1/r_\Omega$, the scaled set $\widetilde \Omega=t\,\Omega$ has unit inradius. Thus, by using the first part of the proof with $\delta/t$ and \eqref{scaling}, we have
\[
T(\Omega;\delta)=t^{-N}\,T\left(\widetilde \Omega;\frac{\delta}{t}\right)>t^{-N}\,\frac{\mathcal{H}^{N-1}(\partial\widetilde\Omega)}{\dfrac{\delta}{t}\,\tanh{\left(\dfrac{\delta}{t}\right)}}=\frac{1}{t}\,\frac{\mathcal{H}^{N-1}(\partial\Omega)}{\dfrac{\delta}{t}\,\tanh(\delta\,r_\Omega)}=\frac{\mathcal{H}^{N-1}(\partial\Omega)}{\delta\,\tanh{(\delta\, r_\Omega)}},
\]
as desired.
\vskip.2cm\noindent
We now come to the proof of the sharpness. Let us consider $\Omega_n:=(-n, n)^{N-1}\times(-1,1)\subseteq\mathbb{R}^N$. By Lemma \ref{lem:hyper} with the choices
\[
\ell_1=\dots=\ell_{N-1}=n\qquad \mbox{ and }\qquad \ell_N=1,
\]
we know that
\[
u_{\Omega_n,\delta}(x)=\frac{\cosh(\delta\,x_N)}{\delta\,\sinh(\delta)}+\frac{1}{\delta\,\sinh(\delta\, n)}\,\sum_{i=1}^{N-1}\cosh(\delta\, x_i),
\]
and
\[
\begin{split}
T(\Omega_n;\delta)&=\sum_{k=1}^{N-1} \left[\frac{1}{\delta\,\tanh(\delta\,n)}\,\mathcal{H}^{N-1}(\Sigma_k)+\sum_{i\not =k} \frac{1}{\delta^2}\,\mathcal{H}^{N-2}(\Sigma_{k,i})\right]\\
&+\left[\frac{1}{\delta\,\tanh(\delta)}\,\mathcal{H}^{N-1}(\Sigma_N)+\sum_{i=1}^{N-1} \frac{1}{\delta^2}\,\mathcal{H}^{N-2}(\Sigma_{N,i})\right]
\end{split},
\]
where we used the notation $\Sigma_k$ and $\Sigma_{k,i}$ from Lemma \ref{lem:hyper}. We now observe that
\[
\mathcal{H}^{N-1}(\Sigma_N)=2^N\, n^{N-1},\qquad \mathcal{H}^{N-1}(\Sigma_k)=2^N\,n^{N-2},\quad \mbox{ for } k=1,\dots,N-1,
\]
\[
\mathcal{H}^{N-2}(\Sigma_{N,i})=2^N\,n^{N-2},\qquad \mbox{ for } i=1,\dots,N-1,
\]
while (this occurs only for $N\ge 3$)
\[
\mathcal{H}^{N-2}(\Sigma_{k,i})=2^{N}\,n^{N-3},\quad \mbox{ for } k=1,\dots,N-1, \ i\in\{1,\dots,N-1\}\setminus\{k\},
\]
and
\[
\mathcal{H}^{N-2}(\Sigma_{k,N})=2^N\,n^{N-2},\quad \mbox{ for } k=1,\dots,N-1.
\]
This finally shows that
\[
T(\Omega_n;\delta)\sim \frac{1}{\delta\,\tanh(\delta)}\,\mathcal{H}^{N-1}(\Sigma_N)=\frac{1}{\delta\,\tanh(\delta)}\,2^N\, n^{N-1},\qquad \mbox{ as } n\to\infty.
\]
As for the measure of the boundary, we have
\[
\mathcal{H}^{N-1}(\partial\Omega_n)=2\,(2\,n)^{N-1}+2\,(N-1)\,(2\,n)^{N-2}\sim 2^N\,n^{N-1},\qquad \mbox{ as } n\to\infty,
\]
while clearly by construction we have $r_{\Omega_n}=1$. By gathering all these informations, we get
\[
\frac{T(\Omega_n;\delta)\,\tanh(\delta\,r_{\Omega_n})}{\mathcal{H}^{N-1}(\partial\Omega_n)}\sim \frac{1}{\delta\,\tanh(\delta)}\,\frac{2^N\, n^{N-1}}{2^N\,n^{N-1}}\,\tanh(\delta)=\frac{1}{\delta},\qquad \mbox{ as } n\to\infty,
\]
as claimed.
\end{proof}
\begin{rem}\label{remark:weakly-superharmonic}
We observe that the convexity assumption in the previous result has been used only to guarantee the property \eqref{lema surface difference}.
Thus Theorem \ref{thm-lower-bound} continues to hold for every open bounded Lipschitz set $\Omega$ which enjoys this property. This is the case, for example, when $N=2$ and $\Omega$ is simply connected or doubly connected, see \cite[Section 2]{Her}. Other sets having property \eqref{lema surface difference} are those for which the distance function $d_\Omega$ is weakly superharmonic in $\Omega$, i.e. such that
\[
\int_\Omega \langle \nabla d_\Omega,\nabla \varphi\rangle\,dx\ge 0,\qquad \mbox{ for every } \varphi\in C^\infty_0(\Omega) \mbox{ with } \varphi\ge 0,
\]
see for example \cite[Remark 3.2]{BrT}.
\par
We recall that on a convex set $\Omega$ the distance $d_\Omega$ is always weakly superharmonic, since it is a concave Lipschitz function. However, the weak superharmonicity of $d_\Omega$ is equivalent to the convexity of $\Omega$ only for dimension $N=2$ (see \cite[Theorem 2]{AU}), while in higher dimension this is a weaker condition. We refer to \cite[Section 5]{AU} for an example of non-convex set with superharmonic distance.
\par
Finally, we also recall that if $\Omega$ has a $C^2$ boundary, the superharmonicity of $d_\Omega$ is equivalent to the fact that the mean curvature of $\partial\Omega$ is non-negative (see \cite{LLL}).
\end{rem}

\subsection{Convex sets: upper bounds}
\label{section:upper-bound}

By appealing to the dual formulation \eqref{dual functional}, we can obtain the following upper bound of geometric flavor, which can be seen as a sort of Steklov version of the so-called {\it Diaz-Weinstein inequality} (see \cite{DW}). For this, we first need to recall the definition of {\it high ridge set} of an open bounded set $\Omega\subset\mathbb{R}^N$. This is given by
\[
M(\Omega):=\Big\{x\in\Omega\,:\, B_{r_\Omega}(x)\subset\Omega\Big\},
\]
i.e. this is the collection of centers of maximal balls inscribed in $\Omega$.
\begin{prop}
	Let $\Omega\subset \mathbb{R}^N$ be an open bounded convex set. Then, we have the following upper bound
	\begin{equation}
	\label{upper_convex}
		T(\Omega;\delta)\leq \frac{1}{r_\Omega^2}\,\left(\mathcal{I}_\#(\Omega)+\frac{N^2}{\delta^2}\,|\Omega|\right),\qquad\mbox{where}\quad \mathcal{I}_\#(\Omega)=\min_{x_0\in M(\Omega)}\int_{\Omega}\lvert x-x_0\rvert^2\,dx.
	\end{equation}
\end{prop}
\begin{proof}
Let $x_0\in M(\Omega)$, then accordingly we have $B_{r_\Omega}(x_0)\subset\Omega$.
We make the following choices
	\begin{equation}
	\label{trial}
		 \phi_0=\frac{x-x_0}{r_\Omega}\qquad\mbox{and}\qquad g_0=\frac{N}{\delta^2\, r_\Omega},
	\end{equation}
and observe that $(\phi_0,g_0)\in \mathcal{A}^+(\Omega)$. Indeed, we clearly have by construction
\[
-\mathrm{div}\phi_0+\delta^2\,g=0,\qquad \mbox{ in }\Omega.
\]
As for the flux condition on the boundary, convexity of $\Omega$ entails that
\[
\langle x-x_0,\nu_{\Omega}\rangle \geq r_\Omega,\qquad \mbox{ for $\mathcal{H}^{N-1}-$a.\,e. } x\in\partial\Omega,
\]	
see for example \cite[Lemma 2.1]{BM}. By appealing to Lemma \ref{lm:dual_formulation} and to the arbitrariness of $x_0\in M(\Omega)$, we finally get the claimed estimate.
\end{proof}
\begin{rem}
By recalling the characterization for extremals of the dual problem \eqref{dual functional}, it is not difficult to see that estimate \eqref{upper_convex} is not sharp. Indeed, the trial pair \eqref{trial} is such that
\[
\phi_0\not=\nabla g.
\]
However, thanks to Theorem \ref{thm:asymptotics}, from \eqref{upper_convex} we recover in the limit the inequality
\[
\frac{(\mathcal{H}^{N-1}(\partial\Omega))^2}{|\Omega|}=\lim_{\delta\to 0^+} \delta^2\,T(\Omega;\delta)\le \frac{N^2\,|\Omega|}{r_\Omega^2},
\]
that is
\[
\mathcal{H}^{N-1}(\partial\Omega)\le \frac{N\,|\Omega|}{r_\Omega},
\]
which is sharp. Equality in the latter is attained for balls, for examples.
\end{rem}
\begin{rem}
The geometric quantity $\mathcal{I}_\#(\Omega)$ defined above is quite similar to the more usual one
\[
\mathcal{I}(\Omega)=\min_{x_0\in \mathbb{R}^N}\int_{\Omega}\lvert x-x_0\rvert^2\,dx,
\]
which is called {\it polar moment of inertia of $\Omega$}. However, in general we have
\[
\mathcal{I}_\#(\Omega)\ge \mathcal{I}(\Omega),
\]
and the inequality is strict, unless the barycenter of $\Omega$, which is the unique minimizer for $\mathcal{I}(\Omega)$, belongs to $M(\Omega)$.
\end{rem}
We now exploit once again the dual formulation \eqref{dual functional}, in order to give a sharp upper bound this time.
At this aim, for an open bounded convex set $\Omega\subset\mathbb{R}^N$, we define its {\it proximal radius} by
\[
L_\Omega:=\inf\Big\{R>0\, :\, \exists x_0\in M(\Omega) \mbox{ such that } \Omega\subset B_R(x_0)\Big\}.
\]
In other words, $L_\Omega$ is the radius of the smallest ball containing $\Omega$ and having center in the high ridge set $M(\Omega)$. We will call {\it proximal center} any point $x_\Omega\in M(\Omega)$, such that
\[
\Omega\subset B_{L_\Omega}(x_\Omega).
\]
Such a point exists and is unique by Lemma \ref{lem:proximal} below.
\begin{figure}
\includegraphics[scale=.3]{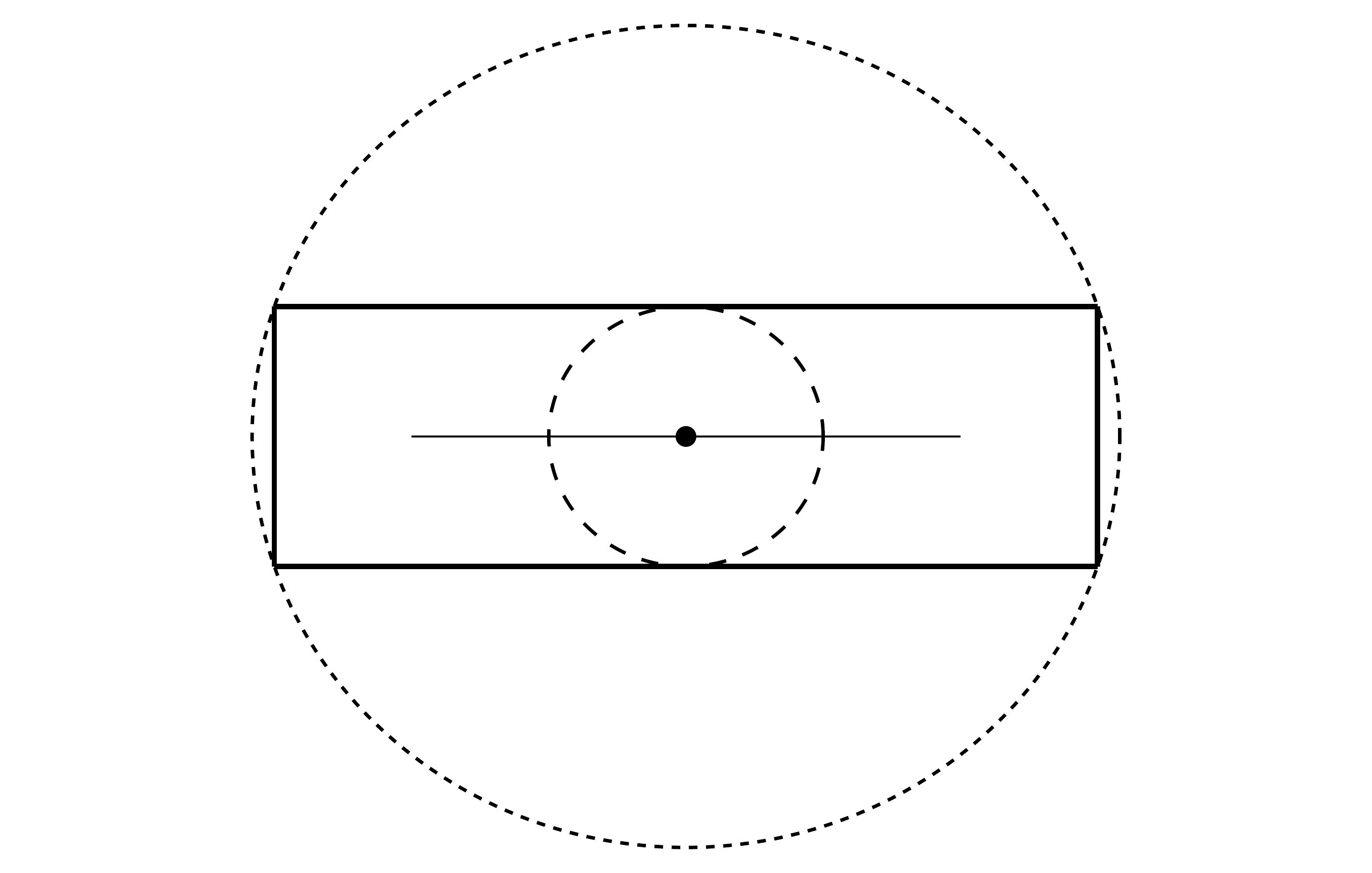}
\caption{The construction of the proximal radius: in bold line, the boundary of the convex set $\Omega$; the thin black line is the high ridge $M(\Omega)$. The black dot is the proximal center, which in this particular example coincides with the center of symmetry of the set.}
\end{figure}
Observe that by construction, for the proximal center we have
\begin{equation}
\label{contained}
B_{r_\Omega}(x_\Omega)\subset \Omega\subset B_{L_\Omega}(x_\Omega).
\end{equation}
This will be crucial for the proof of the following result.
\begin{thm}\label{thm:upper-bound-convex}
Let $\delta>0$ and let $\Omega\subset\mathbb{R}^N$ be an open bounded convex set. With the notation above, we have
\begin{equation}
\label{DWboundary}
T(\Omega;\delta)\le \left(\frac{r_\Omega}{L_\Omega}\right)^{N-2}\,\left(\frac{I_{N/2}(\delta\,L_\Omega)}{I_{N/2}(\delta\,r_\Omega)}\right)^2\,T(B_{L_\Omega};\delta).
\end{equation}
Moreover, equality holds if and only if $\Omega$ is a ball.
\end{thm}
\begin{proof}
 Without loss of generality, we can assume that the proximal center of $\Omega$ coincides with the origin. We set for brevity
\begin{equation}
\label{ugly}
C_{\Omega,\delta}=\mathcal{U}'_{\delta\,L_\Omega}\left(\frac{r_\Omega}{L_\Omega}\right),
\end{equation}
where $\mathcal{U}_\delta$ is the function of Lemma \ref{lem:ball}. We then introduce the pair
\[
\phi_0=\frac{1}{C_{\Omega,\delta}}\,\nabla u_{B_{L_\Omega},\delta}\qquad \mbox{ and }\qquad g_0=\frac{1}{C_{\Omega,\delta}}\,u_{B_{L_\Omega},\delta},
\]
where as always $u_{B_{L_\Omega},\delta}$ is the boundary $\delta-$torsional function of the ball $B_{L_\Omega}$.
We observe that by construction and using \eqref{contained}, we have
\[
-\mathrm{div\,}\phi_0+\delta^2\,g_0=\frac{1}{C_{\Omega,\delta}}\,\left(-\Delta u_{B_{L_\Omega},\delta}+\delta^2\,u_{B_{L_\Omega},\delta}\right)=0,\qquad \mbox{ in }\Omega.
\]
Moreover, since we have
\[
u_{B_{L_\Omega},\delta}(x)=L_\Omega\,u_{B,\delta\,L_\Omega}\left(\frac{x}{L_\Omega}\right)=L_\Omega\,\mathcal{U}_{\delta\, L_\Omega}\left(\frac{|x|}{L_\Omega}\right),
\]
we get that on $\partial\Omega$ it holds
\[
\langle \phi_0,\nu_{\partial\Omega}\rangle=\frac{1}{C_{\Omega,\delta}}\,\frac{1}{|x|}\,\mathcal{U}'_{\delta\, L_\Omega}\left(\frac{|x|}{L_\Omega}\right)\,\langle x,\nu_{\partial\Omega}\rangle\ge \frac{1}{C_{\Omega,\delta}}\,\frac{1}{|x|}\,\mathcal{U}'_{\delta\, L_\Omega}\left(\frac{|x|}{L_\Omega}\right)\,r_\Omega,
\]
$\mathcal{H}^{N-1}-$almost everywhere, where we used again \cite[Lemma 2.1]{BM}. We then recall that
\[
\varrho \mapsto \frac{\mathcal{U}_\delta'(\varrho)}{\varrho},
\]
is monotone non-decreasing by Lemma \ref{lem:ball} and by construction
\[
|x|\ge r_\Omega,\qquad \mbox{ for every }x\in\partial\Omega.
\]
This permits to estimate the flux of $\phi_0$ on the boundary of $\Omega$ by
\[
\langle \phi_0,\nu_{\partial\Omega}\rangle\ge \frac{1}{C_{\Omega,\delta}}\,\mathcal{U}'_{\delta\, L_\Omega}\left(\frac{r_\Omega}{L_\Omega}\right).
\]
Recalling the definition \eqref{ugly} of $C_{\Omega,\delta}$, we finally obtain that $(\phi_0,g_0)\in\mathcal{A}^+(\Omega)$. By Lemma \ref{lm:dual_formulation}, we thus obtain
\begin{equation}
\label{ubenlarge}
\begin{split}
T(\Omega;\delta)&\le \frac{1}{(C_{\Omega,\delta})^2}\,\left[\int_\Omega |\nabla u_{B_{L_\Omega},\delta}|^2\,dx+\delta^2\,\int_\Omega (u_{B_{L_\Omega},\delta})^2\,dx\right]\\
&\le \frac{1}{(C_{\Omega,\delta})^2}\,\left[\int_{B_{L_\Omega}} |\nabla u_{B_{L_\Omega},\delta}|^2\,dx+\delta^2\,\int_{B_{L_\Omega}} (u_{B_{L_\Omega},\delta})^2\,dx\right]=\frac{T(B_{L_\Omega};\delta)}{(C_{\Omega,\delta})^2}.
\end{split}
\end{equation}
We now compute the constant $C_{\Omega,\delta}$
\[
\begin{split}
C_{\Omega,\delta}=\mathcal{U}'_{\delta\,L_\Omega}\left(\frac{r_\Omega}{L_\Omega}\right)&=\frac{d}{d\varrho}\left(\frac{(\delta\,L_\Omega\,\varrho)^{1-N/2}\,I_{N/2-1}(\delta\,L_\Omega\, \varrho)}{(\delta\,L_\Omega)^{2-\frac{N}{2}}\,I_{N/2}(\delta\,L_\Omega)}\right)_{|\varrho=\frac{r_\Omega}{L_\Omega}}\\
&=\left(\frac{\varrho^{1-N/2}\,I_{N/2}(\delta\, L_\Omega\,\varrho)}{I_{N/2}(\delta\,L_\Omega)}\right)_{|\varrho=\frac{r_\Omega}{L_\Omega}}\\
&=\left(\frac{r_\Omega}{L_\Omega}\right)^\frac{2-N}{2}\,\frac{I_{N/2}(\delta\,r_\Omega)}{I_{N/2}(\delta\,L_\Omega)},
\end{split}
\]
where we used \eqref{magic} in order to compute the derivative. This concludes the proof of the inequality.
\vskip.2cm\noindent
Let us now discuss equality cases. At first, it is quite clear that for any ball we get equality in \eqref{DWboundary}. Indeed, in this case, the set $M(\Omega)$ just coincides with the center of the ball and $r_\Omega=L_\Omega$. We now assume that $\Omega\subset\mathbb{R}^N$ is an open bounded convex set for which \eqref{DWboundary} holds as an equality. As before, we can suppose that $x_\Omega=0$. Then the equality sign must hold in all the inequalities above. In particular, equality in \eqref{ubenlarge} entails that we must have
\[
\int_\Omega (u_{B_{L_\Omega},\delta})^2\,dx=\int_{B_{L_\Omega}} (u_{B_{L_\Omega},\delta})^2\,dx.
\]
By virtue of \eqref{contained} and the fact that $u_{B_{L_\Omega},\delta}$ does not vanish, we finally get that
\[
|B_{L_\Omega}\setminus \Omega|=0.
\]
By convexity, this fact and \eqref{contained} in turn imply that $\Omega$ must coincide with the ball $B_{L_\Omega}$.
\end{proof}
From the previous result, we can get the following sharp geometric estimate, involving four geometric quantities.
\begin{cor}\label{cor:geometric}
Let $\delta>0$ and let $\Omega\subset\mathbb{R}^N$ be an open bounded convex set. With the notation above, we have
\begin{equation}
\label{DWisop}
\frac{(\mathcal{H}^{N-1}(\partial\Omega))^2}{|\Omega|}\le N^2\,\omega_N\,\left(\frac{L_\Omega}{r_\Omega}\right)^2\,L_\Omega^{N-2}.
\end{equation}
Equality holds if and only if $\Omega$ is a ball.
\end{cor}
\begin{proof}
We multiply inequality \eqref{DWboundary} by $\delta^2$ and then take the limit as $\delta$ goes to $0$. By using Theorem \ref{thm:asymptotics} and \eqref{Bessel-behavior}, we get
\[
\frac{(\mathcal{H}^{N-1}(\partial\Omega))^2}{|\Omega|}\le \left(\frac{L_\Omega}{r_\Omega}\right)^2\,\frac{(\mathcal{H}^{N-1}(\partial B_{L_\Omega}))^2}{|B_{L_\Omega}|}.
\]
This gives the desired inequality \eqref{DWisop}.
\par
It is straightforward to see that balls give equality in \eqref{DWisop}. On the other hand, let us suppose that $\Omega$ is an open bounded convex set which attains equality in \eqref{DWisop}. In order to prove that $\Omega$ must be a ball, we go back to \eqref{ubenlarge}: with a slightly more careful estimate, we can obtain
\[
\begin{split}
T(\Omega;\delta)&\le \frac{1}{(C_{\Omega,\delta})^2}\,\left[\int_\Omega |\nabla u_{B_{L_\Omega},\delta}|^2\,dx+\delta^2\,\int_\Omega (u_{B_{L_\Omega},\delta})^2\,dx\right]\\
&\le \frac{1}{(C_{\Omega,\delta})^2}\,\left[\int_{B_{L_\Omega}} |\nabla u_{B_{L_\Omega},\delta}|^2\,dx+\delta^2\,\int_{B_{L_\Omega}} (u_{B_{L_\Omega},\delta})^2\,dx\right]\\
&-\frac{\delta^2}{(C_{\Omega,\delta})^2}\,\int_{B_{L_\Omega}\setminus\Omega}(u_{B_{L_\Omega},\delta})^2\,dx\\
&=\frac{T(B_{L_\Omega};\delta)}{(C_{\Omega,\delta})^2}-\frac{\delta^2}{(C_{\Omega,\delta})^2}\,\int_{B_{L_\Omega}\setminus\Omega}(u_{B_{L_\Omega},\delta})^2\,dx.
\end{split}
\]
By recalling the value of $C_{\Omega;\delta}$, we get the following enhanced version of \eqref{DWboundary}
\[
T(\Omega;\delta)\le \left(\frac{r_\Omega}{L_\Omega}\right)^{N-2}\,\left(\frac{I_{N/2}(\delta\,L_\Omega)}{I_{N/2}(\delta\,r_\Omega)}\right)^2\,\left[T(B_{L_\Omega};\delta)-\delta^2\,\int_{B_{L_\Omega}\setminus\Omega}(u_{B_{L_\Omega},\delta})^2\,dx\right].
\]
As in the first part of the proof, we multiply both sides by $\delta^2$ and then take the limit as $\delta$ goes to $0$. By Theorem \ref{thm:asymptotics}, we now get the enhanced version of \eqref{DWisop}
\[
\frac{(\mathcal{H}^{N-1}(\partial\Omega))^2}{|\Omega|}\le \left(\frac{L_\Omega}{r_\Omega}\right)^2\,\left[\frac{(\mathcal{H}^{N-1}(\partial B_{L_\Omega}))^2}{|B_{L_\Omega}|}-\left(\frac{\mathcal{H}^{N-1}(\partial B_{L_\Omega})}{|B_{L_\Omega}|}\right)^2\,|B_{L_\Omega}\setminus\Omega|\right].
\]
Thus, if $\Omega$ attains equality in \eqref{DWisop}, in particular it must result
\[
|B_{L_\Omega}\setminus\Omega|=0,
\]
which implies that $\Omega$ is a ball.
\end{proof}

\appendix

\section{Proximal radius}

We give in this section some elementary facts about the proximal radius of a convex set. We start with the following
\begin{lem}
\label{lem:proximal}
Let $\Omega\subset\mathbb{R}^N$ be an open bounded convex set. Then there exists a unique proximal center.
\end{lem}
\begin{proof}
We first observe that the high ridge set $M(\Omega)$ is a closed convex set. This follows from the fact that it coincides with the level set $\{x\in\overline\Omega\, :\, d_\Omega(x)=r_\Omega\}$ and the distance function is continuous and concave on $\overline\Omega$. We now claim that
\begin{equation}
\label{other}
L_\Omega=\inf_{x\in M(\Omega)} \left(\max_{y\in\partial\Omega} |x-y|\right),
\end{equation}
and that the infimum is attained by a proximal center. Indeed, first observe that the function
\[
x\mapsto \max_{y\in\partial\Omega} |x-y|,
\]
is $1-$Lipschitz, as a supremum of a family of $1-$Lipschitz uniformly bounded functions. By recalling that $M(\Omega)$ is compact, existence of a minimizer follows from Weierstrass' Theorem.
\par
Let us now take $x_0\in M(\Omega)$ to be one of these minimizers. By taking the ball $B_R(x_0)$ with $R=\max_{y\in\partial\Omega} |x_0-y|$, we clearly have
\[
\Omega\subset B_R(x_0),
\]
and thus
\[
L_\Omega\le R=\max_{y\in\partial\Omega} |x_0-y|=\min_{x\in M(\Omega)} \left(\max_{y\in\partial\Omega} |x-y|\right).
\]
In order to prove the reverse inequality, let $r>0$ be such that there exists $x_r\in M(\Omega)$ with the property
\[
\Omega\subset B_r(x_r).
\]
This implies that
\[
|x_r-y|\le r,\qquad \mbox{ for every } y\in\partial\Omega,
\]
and thus
\[
\min_{x\in M(\Omega)} \left(\max_{y\in\partial\Omega} |x-y|\right)\le r.
\]
By taking the infimum over admissible $r>0$, we conclude the proof of \eqref{other}.
\par
The previous part of the proof also shows the existence of at least a proximal center $x_0\in M(\Omega)$. Let us suppose that $x_1\in M(\Omega)$ is another proximal center. This implies that
\[
\Omega\subset B_{L_\Omega}(x_0)\cap B_{L_\Omega}(x_1).
\]
In particular, such an intersection is not empty and thus $L_\Omega>|x_0-x_1|/2$.
By convexity of $M(\Omega)$, the midpoint $\overline{x}=(x_0+x_1)/2$ still belongs to $M(\Omega)$. If we define
\[
\widetilde{R}=\sqrt{L_\Omega^2-\frac{|x_0-x_1|^2}{2}},
\]
we have $0<\widetilde{R}<L_\Omega$ and by construction
\[
\Omega\subset B_{L_\Omega}(x_0)\cap B_{L_\Omega}(x_1)\subset B_{\widetilde R}(\overline{x}).
\]
This violates the minimality of $L_\Omega$, thus giving the desired contradiction.
\end{proof}
The proximal radius is actually comparable to more familiar geometric quantities, like the diameter $\mathrm{diam}(\Omega)$ and the {\it circumradius} $R_\Omega$. The latter is defined by
\[
R_\Omega=\inf\Big\{R>0\, :\, \exists x\in \mathbb{R}^N \mbox{ such that } \Omega\subset B_R(x)\Big\}.
\]
This is the radius of the smallest ball entirely containing $\Omega$.
\begin{lem}
Let $\Omega\subset\mathbb{R}^N$ be an open bounded convex set. Then we have
\[
R_\Omega\le L_\Omega< \mathrm{diam}(\Omega),
\]
and both inequalities are sharp.
\end{lem}
\begin{proof}
The first inequality $R_\Omega\le L_\Omega$ is a trivial consequence of the definitions. We also observe that we have equality for every convex set such that $R_\Omega$ is attained by a ball centered on $M(\Omega)$: this happens for example for every convex sets having $N$ orthogonal axis of symmetry.
\par
For the second inequality, for every $y\in\partial\Omega$ and every $x\in M(\Omega)$ we have
\[
|x-y|<\mathrm{diam}(\Omega),
\]
by definition of diameter. Observe that the inequality is strict, since $M(\Omega)$ is at a distance $r_\Omega$ (i.e. the inradius) from the boundary $\partial\Omega$, thus its points can not be extremal for the diameter.
By recalling \eqref{other}, this gives the desired inequality.
\par
In order to prove sharpness of this estimate, we confine ourselves for simplicity to $N=2$. We take the sequence of right triangles $\{T_n\}_{n\in\mathbb{N}}$ having vertices at
\[
P_1=(-1,0),\qquad P_2=(1,0),\qquad P_3=\left(\cos\left(\frac{1}{n+1}\right),\,\sin\left(\frac{1}{n+1}\right)\right).
\]
We clearly have that $\mathrm{diam}(T_n)=2$. In order to compute $L_{T_n}$, we observe that the high ridge set is a singleton
\[
M(T_n)=\left\{(\alpha_n,\beta_n)\right\},\qquad \mbox{ with } \lim_{n\to\infty} \alpha_n=1,\quad \lim_{n\to\infty} \beta_n=0.
\]
Thus we have
\[
\lim_{n\to\infty}L_{T_n}=\lim_{n\to\infty} |(\alpha_n,\beta_n)-P_1|=\lim_{n\to\infty} \sqrt{(\alpha_n+1)^2+\beta_n^2}=2,
\]
as desired.
\end{proof}

\end{document}